\documentclass[10pt,twoside]{siamart1116}

\usepackage[english]{babel}
\usepackage{graphicx,epstopdf,epsfig}
\usepackage{amsfonts,epsfig,fancyhdr,graphics, hyperref,amsmath,amssymb}

\usepackage{algorithm}
\usepackage{algpseudocode}
\usepackage{algorithmicx}
\usepackage{setspace}

\usepackage{relsize,exscale}
\usepackage{array}
\usepackage{subfig}
\usepackage[shortcuts]{extdash}

\newcommand{\HE}{Name of Handling Editor}
\newcommand{\DoS}{Month/Day/Year}
\newcommand{\DoA}{Month/Day/Year}
\newcommand{\CA}{Marcel Schweitzer}

\newcommand{\Names}{Bahar Arslan, Samuel D.\ Relton, Marcel Schweitzer}
\newcommand{\Title}{Structured level\=/2 condition numbers of matrix functions}

\newtheorem{remark}[theorem]{Remark}

\renewcommand{\theequation}{\thesection.\arabic{equation}}
\renewtheorem{theorem}{Theorem}[section]

\def\vec{\textup{vec}}
\def\J{\mathbb{J}}
\def\L{\mathbb{L}}
\def\G{\mathbb{G}}
\def\S{\mathbb{S}}
\def\R{\mathbb{R}}
\def\C{\mathbb{C}}
\def\K{\mathbb{K}}
\def\M{\mathcal{M}}
\def\N{\mathcal{N}}

\def\adj{{\textstyle\star}}
\def\F{Fr\'{e}chet }
\def\nbyn{n \times n}
\def\struc{\textup{struc}}
\def\cond{\mathrm{\mathop{cond}}}

\def\scalar#1{\langle #1\rangle}
\newcolumntype{P}[1]{>{\centering\arraybackslash}p{#1}}

\makeatletter
\def\revddots{\mathinner{\mkern1mu\raise\p@
    \vbox{\kern7\p@\hbox{.}}\mkern2mu
    \raise4\p@\hbox{.}\mkern2mu\raise7\p@\hbox{.}\mkern1mu}}
\makeatother

\makeatletter
\def\mymatrix#1{\null\,\vcenter{\normalbaselines\m@th
    \ialign{\hfil$##$\hfil&&\quad\hfil$##$\hfil\crcr
      \mathstrut\crcr\noalign{\kern-\baselineskip}
      #1\crcr\mathstrut\crcr\noalign{\kern-\baselineskip}}}\,}
\makeatother
\def\mybmatrix#1{\left[ \mymatrix{#1} \right]}

\makeatletter
\def\revddots{\mathinner{\mkern1mu\raise\p@
    \vbox{\kern7\p@\hbox{.}}\mkern2mu
    \raise4\p@\hbox{.}\mkern2mu\raise7\p@\hbox{.}\mkern1mu}}
\makeatother

\mathcode`@="8000 
 {\catcode`\@=\active\gdef@{\mkern1mu}}
\newcounter{mylineno}
\makeatletter
\let\oldtabcr\@tabcr

\def\mynewline{\refstepcounter{mylineno}%
                \llap{\footnotesize\arabic{mylineno}\hspace{5pt}}%
               }

\gdef\@tabcr{\@stopline \@ifstar{\penalty%
            \@M \@xtabcr}\@xtabcr\mynewline}

\makeatother

\newcounter{exprmt}
\def\Experiment{Experiment}

\newenvironment{exprmt}%
               {\begin{list}{\indent\emph{\Experiment} {\arabic{section}.\arabic{exprmt}.}}%
                {\usecounter{exprmt}
                \setlength{\leftmargin}{\rightmargin}
                \setlength{\labelwidth}{\leftmargin}
                \addtolength{\labelwidth}{-\labelsep}
                \setlength{\topsep}{0in}
                \setlength{\itemsep}{0pt}
                \setlength{\listparindent}{\parindent}%
               }}%
               {\end{list}}

\DeclareMathOperator{\nnz}{nnz}

\begin{document}

\bibliographystyle{plain}

\setcounter{page}{1}

\thispagestyle{empty}

 \title{\Title\thanks{Received
 by the editors on \DoS.
 Accepted for publication on \DoA. 
 Handling Editor: \HE. Corresponding Author: \CA}}

\author{
Bahar Arslan\thanks{Mathematics Department, Faculty of Engineering and Natural Sciences, Bursa Technical
University, Yıldırım/Bursa, Turkey,  \texttt{bahar.arslan@btu.edu.tr}}
\and
Samuel D.\ Relton\thanks{Faculty of Medicine and Health, University of Leeds, Leeds, UK, \texttt{s.d.relton@leeds.ac.uk}}
\and 
Marcel Schweitzer\thanks{School of Mathematics and Natural Sciences, Bergische Universit\"at Wuppertal, 42097 Wuppertal, Germany, \texttt{marcel@uni-wuppertal.de}}}

\pagestyle{myheadings}
\markboth{\Names}{\Title}

\maketitle

\begin{abstract}
Matrix functions play an increasingly important role in many areas of scientific computing and engineering disciplines. In such real-world applications, algorithms working in floating-point arithmetic are used for computing matrix functions and additionally, input data might be unreliable, e.g., due to measurement errors. Therefore, it is crucial to understand the sensitivity of matrix functions to perturbations, which is measured by condition numbers. However, the condition number itself might not be computed exactly as well due to round-off and errors in the input. The sensitivity of the condition number is measured by the so-called \emph{level\=/2 condition number}. For the usual (level-1) condition number, it is well-known that \emph{structured} condition numbers (i.e., where only perturbations are taken into account that preserve the structure of the input matrix) might be much smaller than unstructured ones, which, e.g., suggests that structure-preserving algorithms for matrix functions might yield much more accurate results than general-purpose algorithms. In this work, we examine structured level\=/2 condition numbers in the particular case of restricting the perturbation matrix to an automorphism group, a Lie or Jordan algebra or the space of quasi-triangular matrices. In numerical experiments, we then compare the unstructured level\=/2 condition number with the structured one for some specific matrix functions such as the matrix logarithm, matrix square root, and matrix exponential.
\end{abstract}

\begin{keywords}
level\=/2 condition number, matrix function, automorphism group, Lie algebra, Jordan algebra, quasi-triangular matrices
\end{keywords}

\begin{AMS}
65F35, 15A16, 65F60
\end{AMS}

\section{Introduction}
Matrix functions $f:\K^{\nbyn}\rightarrow \K^{\nbyn}$,
where $\K=\C$ or $\R$,
such as the matrix exponential, logarithm, and square root are used in a wide variety of applications including the solution of differential equations by exponential integrators, network analysis and machine learning~\cite{app7,app4,app5,app6,app2,app3,app1}. The condition number measures the sensitivity of a matrix function to perturbations in the input due to rounding or measurement error. 

The (level-1) absolute condition number is
defined for a matrix function $f:\K^{\nbyn}\rightarrow \K^{\nbyn}$ as
\begin{equation}
\label{eq.condnum}
\cond^{[1]} (f,A) = \lim_{\epsilon\rightarrow 0}\:\:\sup_{\| E\|\leq \epsilon} \dfrac{\| f(A+E)-f(A)\|}{\epsilon}.
\end{equation}
It can be computed via the \F derivative $L^{(1)}_f(A,\ \cdot\ )$ of the matrix function, a linear function in the perturbation matrix $E$ such that $L^{(1)}_f(A, E)$ satisfies~\cite[Section~3]{high:FM}
\begin{equation*}
\|f(A+E)-f(A)-L^{(1)}_f(A,E)\|=o(\| E\|),
\end{equation*}
as we outline at the beginning of section~\ref{subsec.unstruc_lvl2}. When the input matrix $A$ has some structure described by a Lie or Jordan algebra (e.g.~symmetric or skew-Hermitian), and we consider only perturbations that retain this structure, the resulting \emph{structured level-1 condition number} can be significantly smaller. This has been investigated in~\cite{Tisseur_scond} for linear systems, matrix inversion, and distance to singularity. The effect of structured perturbations of matrix functions in Lie and the Jordan algebras is also studied by Davies~\cite{Davies:2004:SCM}. The structured level-1 condition number using differentials between smooth manifolds is also analysed for automorphism groups in~\cite{bvt2019}.  

Unfortunately, the condition number itself might not be computed precisely: previous work on the condition number of the condition number (the so-called \emph{level\=/2 condition number}) \cite{hire14a} has shown that it can be highly sensitive in some cases. Similar findings are also known for the level\=/2 condition number for solving linear systems~\cite{Demmel_level2,Des_Higham1995193}.

The aim of this work is to investigate the structured level\=/2 condition number of matrix functions by enforcing structure on the perturbation matrices and it is organised as follows. In section~\ref{sec.prev_work} we summarise existing work on the unstructured level\=/2 condition number of matrix functions and the structured level-1 condition number. We then build upon these foundations in section~\ref{sec:new_results} and derive upper bounds for structured condition numbers of matrix functions of different structures. We first investigate the case of matrices coming from Lie and Jordan algebras or automorphism groups and then focus on the important case of quasi-triangular matrices. Additionally, for a few specific choices of function and matrix class, we derive exact formulas for the structured level\=/2 condition number. Numerical experiments in section~\ref{sec.numexp} provide a thorough comparison of the structured and unstructured condition numbers for the  matrix exponential, logarithm, and square root over a variety of structures.
We summarise our findings in section~\ref{sec.conclusions} and give some ideas for future work in this area.

\section{Previous work on matrix function conditioning}
\label{sec.prev_work}

In this section we summarise the current theory on matrix function conditioning, which we build upon to derive corresponding results for the structured level\=/2 condition number in later sections. In particular, we focus here on the structured level-1 condition number and the unstructured level\=/2 condition number. We reproduce some of the derivations in quite a bit of detail, as the concepts and notations introduced there are essential for understanding the derivations of our new results in section~\ref{sec:new_results}.

\subsection{The unstructured level\=/2 condition number case}
\label{subsec.unstruc_lvl2}

In this section we summarise previous results for the unstructured level\=/2 condition number,
largely based upon \cite{high:FM} and \cite{hire14a}. To begin, the level-1 condition number is the norm of the first order \F derivative operator $L^{(1)}_f(A,\ \cdot \ )$, i.e.,
$$
\cond^{[1]}(f,A)= \max_{E\neq 0}\dfrac{\| L^{(1)}_f(A,E)\|}{\| E\|}=\max_{\|E\|=1}\| L^{(1)}_f(A,E)\|,
$$
where $\|\cdot\|$ can in principle be any matrix norm. It is convenient to use the Frobenius norm $\|\cdot\|_F$ because of the fact that $\| A\|_{F}=\| \vec(A)\|_{2}$, where $\vec(\cdot)$ stacks the columns of a matrix into one long vector. This reduces the problem to finding the 2-norm of $K^{(1)}_f(A)$, the \emph{Kronecker form} of the \F derivative: The Kronecker form is the matrix of size $n^2 \times n^2$ that satisfies
$$
\vec(L_f^{(1)}(A,E)) = K^{(1)}_f(A)\vec(E) \text{ for all } E \in \K^{\nbyn}.
$$
Using this definition, the condition number in the Frobenius norm can be rewritten as
\[
\cond^{[1]} (f,A)=\max_{E\neq 0} \dfrac{\Vert L^{(1)}_f(A,E)\|_F}{\| E\|_F}=\max_{E\neq 0} \dfrac{\Vert \vec(L^{(1)}_f(A,E))\|_2}{\| \vec(E)\|_2} =\max_{E\neq 0} \dfrac{\| K^{(1)}_f(A)\vec(E)\|_2}{\| \vec(E)\|_2}=\| K^{(1)}_f(A)\|_2. 
\]

Higher order derivatives of matrix functions are defined similarly, and in particular the second \F derivative $L_f^{(2)}(A,E_1,E_2)$ is a multilinear function which satisfies
\begin{equation}
\label{eq.secondfrechet}
\|L_f^{(1)}(A+E_2,E_1)-L_f^{(1)}(A,E_1)-L_f^{(2)}(A,E_1,E_2)\|=o(\|E_2\|).
\end{equation}
Sufficient conditions for the existence of the second \F derivative are given by Higham and Relton~\cite[Theorem~3.5]{hire14a}. This theorem also states that $L^{(2)}_f(A, E_1, E_2)$ is equal to the upper right $\nbyn$ block of $f(X_2)$,
\begin{equation}\label{eq.higherfrechet}
L^{(2)}_f(A, E_1, E_2) = [f(X_2)]_{1:n,1:n},
\end{equation}
where $X_2$ is built using the following recursion
\begin{align*}
X_0&=A,\\
X_k&=I_2\otimes X_{k-1}+
\begin{bmatrix}
  0 & 1 \\
  0 & 0
\end{bmatrix}\otimes I_{2^{k-1}}\otimes E_k, \quad k \geq 1
\end{align*}
with $I_n$ the $\nbyn$ identity matrix and $\otimes$ the Kronecker product.

Using these building blocks, one can define higher-order condition numbers of matrix functions which capture the sensitivity of the condition number itself to perturbations in $A$. We outline the basic approach here, following largely the derivation in~\cite{hire14a}. We begin by defining the absolute level\=/2 condition number via the following equation,
\begin{equation}
\label{eq.deflevel2}
\cond^{[2]}(f,A)=\lim_{\epsilon\rightarrow 0}\:\:\sup_{\| Z\|\leq \epsilon}\, \dfrac{|\cond^{[1]}(f,A+Z)-\cond^{[1]}(f,A)|}{\epsilon}.
\end{equation}
We can derive an upper bound for the level\=/2 condition number (see~\cite{hire14a}) 
by combining equations \eqref{eq.condnum} and \eqref{eq.secondfrechet}, which yields
$$
\cond^{[1]}(f,A+Z)=\max_{\|E\|=1}\|L^{(1)}_f(A,E)+L^{(2)}_f(A,E,Z)+o(\|Z\|) \|.
$$
From the triangular inequality we get the upper bound
\begin{align*}
|\cond^{[1]}(f,A+Z)-\cond^{[1]}(f,A)| & = \left|\max_{\|E\|=1}\|L^{(1)}_f(A,E)+L^{(2)}_f(A,E,Z)+o(\|Z\|) \|-  \max_{\|E\|=1}\|L^{(1)}_f(A,E)\| \right|  \\
& \leq \max_{\|E\|=1}\| L^{(2)}_f(A,E,Z) +o(\|Z\|)\|.
\end{align*}
Using this within the definition of the level\=/2 condition number (\ref{eq.deflevel2}) yields
\begin{align}
\cond^{[2]}(f,A)&\leq \lim_{\epsilon\rightarrow 0}\, \sup_{\| Z\|\leq \epsilon}\,\max_{\|E\|=1}\|L^{(2)}_f(A,E,Z/\epsilon)+o(\|Z\|)/\epsilon  \|\nonumber\\
& = \sup_{\| Z\|\leq 1}\max_{\|E\|=1} \| L^{(2)}_f(A,E,Z) \| \nonumber\\
& =\max_{\| Z\|=1}\max_{\|E\|=1} \| L^{(2)}_f(A,E,Z) \|,\label{eq:bound_unstruc_lvl2}
\end{align}
where the supremum can be replaced with a maximum in a 
finite dimensional vector space and since $L^{(2)}_f(A,E,Z)$ is linear in $Z$ we can have $\| Z\|=1$. Furthermore, as $L_f^{(2)}(A,E,Z)$ is also linear in $E$, there exists a matrix $K^{(1)}_f(A,Z) \in \C^{n^2\times n^2}$ such that
\[
 \vec (L^{(2)}_f(A,E,Z))=K^{(1)}_f(A,Z)\vec(E).
\]
Using these facts and taking the Frobenius norm in the upper bound~\eqref{eq:bound_unstruc_lvl2} leads to
\begin{align}
\cond^{[2]}(f,A)&\leq \max_{\| Z\|_F=1} \max_{\|E\|_F=1} \| L^{(2)}_f(A,E,Z) \|_F \nonumber\\
&=\max_{\| Z\|_F=1} \max_{\|\vec(E)\|_2=1} \|K^{(1)}_f(A,Z)\vec(E) \|_2 \nonumber\\
& =\max_{\| Z\|_F=1}\|K^{(1)}_f(A,Z)\|_2 \nonumber\\
& \leq \max_{\| Z\|_F=1}\|K^{(1)}_f(A,Z)\|_F \nonumber\\
& = \max_{\|\vec(Z)\|_2=1}\|K^{(2)}_f(A)\vec(Z)\|_2\nonumber\\
& = \|K^{(2)}_f(A)\|_2\label{eq.upper_bound_unstructured},
\end{align}
where $K^{(2)}_f(A)$ is the Kronecker form of the second \F derivative. 

We stress that---in contrast to the level-1 condition number---the spectral norm of the Kronecker form only gives an upper bound for the level\=/2 condition number instead of its precise value. The exact value of the level\=/2 condition number can only be determined in certain special cases. Higham and Relton find formulas for the exact level\=/2 condition number for the matrix inverse~\cite[Section~5.2]{hire14a} and the exponential of normal matrices~\cite[Section~5.1]{hire14a}, while Schweitzer finds an exact formula for certain functions of Hermitian matrices which have a strictly monotonic derivative~\cite[Section~3]{schweitzer2023}.

\subsection{The structured case}
\label{subsec.structured_case}
In this subsection we give a brief recap of the framework within which
structured perturbations can be analysed, and proceed to derive similar results to the above for structured matrices.

Let $\S$ be a smooth square matrix manifold $\S\subseteq\K^{\nbyn}$ for a field $\K$. 
For a smooth manifold $\S$, the element $E\in\K^{\nbyn}$ is a tangent vector of 
$\S$ at $A\in\S$ if there is a smooth curve $\gamma:\K\rightarrow \S$ such that 
$\gamma(0)=A$, $\gamma^\prime(0)=E$. 
The set 
\begin{equation}\label{eq:tangent_space}
T_A\S:=\{E\in\K^{\nbyn} \:|\:
        \mbox{$\exists\ \gamma:\K\rightarrow \S$ smooth with
                $\gamma(0)=A$, $\gamma'(0)=E$}\}
\end{equation}
is called the tangent space of $\S$ at $A$.
Furthermore we define a scalar product, 
from $\K^n\times \K^n$ to $\K$: $(x,y)\mapsto \scalar{x,y}_M$ for any nonsingular matrix $M \in \K^{n\times n}$ by, 
$$
\scalar{x,y}_M =\left\{\begin{array}{ll}
x^TMy, & \text{for real or complex bilinear forms,}\\
x^*My, & \text{for sesquilinear forms.}
\end{array}\right.
$$
Finally, 
for any matrix $A\in\K^{\nbyn}$, 
there exists a unique $A^{\adj} \in\K^{\nbyn}$, called the adjoint of $A$ with respect to
the scalar product $\scalar{.,.}_M$, which is given by
$$
A^{\adj}= \left\{\begin{array}{ll}
M^{-1}A^TM, & \text{for real or complex bilinear forms,}\\
M^{-1}A^*M, & \text{for sesquilinear forms.}
\end{array}\right.
$$        
We are interested in three classes of structured matrices associated with  $\scalar{.,.}_M$: a Jordan algebra $\J$, a Lie algebra $\L$ and an automorphism group $\G$ defined by
$$
\J : = \{ A\in\K^{\nbyn} \:|\:A^\adj=A\},\qquad
\L : = \{ A\in\K^{\nbyn} \:|\:A^\adj=-A\}, \qquad
\G: = \{ A\in\K^{\nbyn} \:|\:A^\adj=A^{-1}\},
$$ 
respectively. Practically important choices of $M$ are given by the three matrices
\begin{equation}\label{eq:M_choices}
\Sigma_{p,q}=
\begin{bmatrix}
I_p &  0 \\
0 & -I_q 
\end{bmatrix},
\quad (\text{where } p+q=n), \qquad 
R= \left[
      \begin{array}{c@{\mskip8mu}c@{\mskip8mu}c}
      & & 1         \\[-2pt]
      & \revddots & \\[-4pt]
      1 & &         \\[-2pt]
      \end{array}\right],
  \quad \textup{and} \qquad  
      J=\begin{bmatrix}
0  & I_{n/2}\\
 -I_{n/2}&  0
\end{bmatrix},
\end{equation}
as $\J$, $\L$ and $\G$ then correspond to some commonly encountered matrix structures; cf.~Table~\ref{tab.structures}. Note that for all choices of $M$ from~\eqref{eq:M_choices}, we trivially have that $M=\mu M^*$ with $\mu\in\{+1,-1\}$.

\begin{table}
\def\arraystretch{1.4}
\centering
\caption{Choices for $M$ leading to well-known structures when 
constructing the Jordan algebra, Lie algebra, or automorphism groups $\J$, $\L$ and $\G$.}
\begin{tabular}{|c| p{4.6cm} | p{4.6cm} |p{4cm}|}
\hline
 \multicolumn{4}{|c|}{\textbf{Real/Complex Bilinear forms}}\\
 \hline\hline
 $\textbf{M}$ &\textbf{Automorphism Group} ($\G$)
 &\textbf{Jordan\quad Algebra} ($\J$)&\textbf{Lie Algebra} ($\L$) \\ \hline
 $I_n$& Real orthogonals & Symmetrics & Skew-symmetrics  \\  \hline
$I_n$&Complex orthogonals &Complex symmetrics & Complex\:skew-symmetrics  \\ \hline
 $\Sigma_{p,q}$& Pseudo-orthogonals& Pseudo-symmetrics & Pseudo\:skew-symmetrics \\ \hline
 $\Sigma_{p,q}$& Complex\:pseudo-orthogonals&Complex\:pseudo-symmetrics &Complex\:pseudo-skew-symmetrics \\ \hline
 $R$& Real perplectics & Persymmetric & Perskew-symmetrics \\ \hline
$J$ & Real symplectics & Skew-Hamiltonians & Hamiltonians \\ \hline
 $J$ &Complex symplectics & $J$-skew-symmetric & $J$-symmetrics \\ \hline
 \end{tabular}
 \label{tab.structures}
\end{table}

With these fundamentals, we are ready to define the structured level-1 condition number using differentials between tangent spaces to smooth square matrix manifolds, closely following~\cite{bvt2019}.

\begin{definition}
Let $f : \S_{\M} \rightarrow \S_{\N}$ be a smooth matrix function between two smooth square matrix manifolds $\S_{\M}\subseteq\K^{\nbyn}$ and $\S_{\N}\subseteq\K^{\nbyn}$. Let $A \in \S_{\M}$ be such that $f(A) \in \S_{\N}$ is
defined. Then the absolute structured level-1 condition number of $f(A)$ is
\begin{equation}\label{eq.defstruc}
\cond^{[1]}_{\struc}(f,A) =\lim_{\epsilon \rightarrow 0} \sup_{\substack{\|Y-A\| \leq \epsilon \\ Y \in \S_{\M}}} \dfrac{\| f(Y) - f(A)  \|}{\epsilon}.
\end{equation}
\end{definition}
The enforced condition $Y \in \S_{\M}$ restricts the choice of the perturbation $Y-A$ to
a smaller set than in the unstructured case, so from the definition of the supremum, we get
$$
        \cond^{[1]}_{\struc}(f,A) \leq \cond^{[1]}(f,A).
$$
The structured condition number can equivalently be expressed in terms of the \emph{differential} of $f$. Precisely, if $f:\S_{\M}\rightarrow \S_{\N}$ is $\K$-differentiable, then the differential of $f$ at the point $A$ is the map
$$
df_A:T_A\S_{\M} \rightarrow T_{f(A)}\S_{\N}, \quad df_{A}(\gamma^\prime(0))=(f\circ \gamma)^\prime(0),
$$ 
where $T_A\S_{\M}$ and $T_{f(A)}\S_{\N}$ are tangent spaces as defined in~\eqref{eq:tangent_space}. Using this definition, it holds
\begin{equation}\label{eq:struclvl1_differential}
\cond^{[1]}_{\struc}(f,A)=\|df_A\|,
\end{equation}
where
$$
        \| df_A\| :=
        \max_{\substack{E \in T_A \S_{\M} \\ E\neq 0}}
        \dfrac{\|df_A(E)\|}{\| E\|}.
$$
The equality~\eqref{eq:struclvl1_differential} is proven in~\cite[Theorem~2.3]{bvt2019}. 

For practically computing the structured level-1 condition number, one constructs a structured equivalent of the Kronecker form: Let the columns of $B_{\M}$ span the tangent space $T_A \S_{\M}$ and let $y\in \K^p$ with $p=\dim_{\K}T_A \S_{\M}$.
Then for $E\in T_A \S_{\M}$ we have $\vec(E)=B_{\M}y$. 
Focusing on the Frobenius norm as before,
the structured level-1 condition number can be expressed as
\[
        \cond^{[1]}_{\struc}(f,A)
        =\max_{\substack{E \in T_A \S_{\M} \\ E\neq 0}} \dfrac{\Vert \vec(L_f(A,E))\rVert_2}{\lVert \vec(E)\rVert_2}
        =\max_{\substack{y\in \K^p \\ y\neq 0}} \dfrac{\lVert K_f(A)B_{\M}y\rVert_2}{\lVert B_{\M}y\rVert_2}
        =\lVert K_f(A)B_{\M}B_{\M}^+\rVert_2,
\]
where $B_{\M}^+$ denotes the Moore--Penrose pseudoinverse of $B_{\M}$.
The construction of the tangent space basis $B_{\M}$ is given in \cite[Section~3.1]{bvt2019}, which shows that
for $\S_{\M}\in\{\J,\L\}$ we can construct the basis of the tangent space of $\J$ and $\L$ as 
\begin{equation}
\label{eq.basejordanlie}
B_{\M}=(I_n\otimes M^{-1})D_{\S_\M}
\end{equation}
with
\[
D_{\S_\M}=\left\{\begin{array}{ll}
        \mybmatrix{\widetilde{D}_{\mu s} & \check{I}} & \mbox{if $ \mu s=1$,}\\
        \ \ \widetilde{D}_{\mu s} & \mbox{if $\mu s=-1$,}
        \end{array}\right.
\]
where $s=1$ if $\S=\J$, $s=-1$ if $\S=\L$ and $\widetilde{D}_{\mu s}\in\R^{n^2\times n(n-1)/2}$ is a matrix with the columns
\[
(e_{(i-1)n+j}+\mu s e_{(j-1)n+i})/\sqrt{2}, \quad 1\leq i<j\leq n
\]
and $\check{I}\in\R^{n^2\times n}$ has the column vectors $e_{(i-1)n+i}$, $i=1,\dots,n$. 

For an automorphism group the basis of the tangent space to $\S_{\M} = \G$ at $A$ is given by
\begin{equation}
\label{eq.baseauto}
B_{\M}=(I_n\otimes AM^{-1})D_{\L_\M};
\end{equation}
cf.~\cite[Section~3.2]{bvt2019}.

\section{New results}\label{sec:new_results}
In this section, we derive our main result, Lemma~\ref{lem.struclvl2}, and use it to find upper bounds for structured level\=/2 condition numbers for various different matrix structures in section~\ref{subsec:struc_lvl2} and~\ref{sec.quasitriangular}. In section~\ref{subsec:exact} we discuss two particular cases in which it is possible to exactly compute the structured level\=/2 condition number. In section~\ref{subsec.lower_bounds}, we also comment on how to numerically compute \emph{lower} bounds for level\=/2 condition numbers, which are required for judging the quality of the upper bounds we obtained. 

\subsection{The level\=/2 condition number of matrices from Lie and Jordan algebras or automorphism groups}\label{subsec:struc_lvl2}

For a smooth matrix function $f : \S_{\M} \rightarrow \S_{\N}$ we can define the structured level\=/2 condition number similarly to~\eqref{eq.deflevel2}, but restricting the perturbation matrix $Z$ to $\S_{\M}$,
\begin{equation}\label{eq.defstruclevel2}
\cond_{\struc}^{[2]}(f,A)=\lim_{\epsilon\rightarrow 0}\:\:\sup_{\substack{A+Z\in \S_{\M}\\ \| Z\|\leq \epsilon}}\, \dfrac{|\cond_{\struc}^{[1]}(f,A+Z)-\cond_{\struc}^{[1]}(f,A)|}{\epsilon}.
\end{equation}
The following lemma provides an upper bound of the structured level\=/2 condition number of $f(A)$. 
\begin{lemma}\label{lem.struclvl2}
Let $f : \S_{\M} \rightarrow \S_{\N}$ be a smooth matrix function between two smooth square matrix manifolds $\S_{\M}\subseteq\K^{\nbyn}$ and $\S_{\N}\subseteq\K^{\nbyn}$. Let $A \in \S_{\M}$ be such that $f(A) \in \S_{\N}$ is
defined. Then an upper bound for the structured level\=/2 condition number of $f(A)$ is given by
\begin{equation}\label{eq.upper_bound_structured}
\cond_{\struc}^{[2]}(f,A) \leq \|(B_{\M}B_{\M}^+\otimes I_{n^2}) K^{(2)}_f(A)B_{\M}B^+_{\M}\|_2 
\end{equation}
where the columns of $B_{\M}$ span $T_A \S_{\M}$, the tangent space of $\S_{\M}$ at $A$.
\end{lemma}
\begin{proof}
Using the first and the second \F derivatives gives
$$
\cond_{\struc}^{[1]}(f,A+Z)=\max_{\substack{E\in \S_{\M}, Z\in \S_{\M},\\ \|E\|=1}}\|L^{(1)}_f(A,E)+L^{(2)}_f(A,E,Z)+o(\|Z\|) \|.
$$
Enforcing the structure on the perturbation matrices $E$ and $Z$ with $E,Z\in\S_{\M}$ yields the following upper bound of the structured level\=/2 condition number:
\begin{align}
\cond_{\struc}^{[2]}(f,A)&\leq \max_{\substack{Z\in T_A \S_{\M}\\ \|Z\|_F=1}}\max_{\substack{E\in T_A \S_{\M}\\ \|E\|_F=1}} \| L^{(2)}_f(A,E,Z) \|_F\nonumber\\
& = \max_{\substack{Z\in T_A\S_{\M}\\ \|Z\|_F=1}}\max_{\substack{E\in T_A\S_{\M}\\\|\vec(E)\|_2=1}} \|K^{(1)}_f(A,Z)\vec(E) \|_2 \nonumber\\
& = \max_{\substack{Z\in T_A\S_{\M}\\ \|Z\|_F=1}}\max_{\substack{y_1\in \K^p\\ y_1\neq 0}} \|K^{(1)}_f(A,Z)B_{\M}y_1 \|_2 \nonumber\\
& = \max_{\substack{Z\in T_A\S_{\M}\\ \|Z\|_F=1}}\|K^{(1)}_f(A,Z)B_{\M}B_{\M}^+ \|_2,\label{eq:proof_part1}
\end{align}
where, for the last equality, we used the fact that $B_{\M}=B_{\M}B_{\M}^+B_{\M}$ together with the definition of the 2-norm,
$$
\|K^{(1)}_f(A,Z)B_{\M}B_{\M}^+ \|_2=\max_{\substack{y_1\in \K^p\\ y_1\neq 0}}\dfrac{\|K^{(1)}_f(A,Z)B_{\M}B_{\M}^+B_{\M}y_1 \|_2}{\|B_{\M}y_1\|_2}.
$$
Bounding the 2-norm by the Frobenius norm, we further obtain, starting from~\eqref{eq:proof_part1},
\begin{align*}
\cond_{\struc}^{[2]}(f,A)&\leq \max_{\substack{Z\in T_A\S_{\M}\\ \|Z\|_F=1}}\|K^{(1)}_f(A,Z)B_{\M}B_{\M}^+ \|_F \\
& =\max_{\substack{Z\in T_A\S_{\M}\\\|\vec(Z)\|_2=1}} \|(B_{\M}B_{\M}^+\otimes I_{n^2}) K^{(2)}_f(A)\vec(Z)\|_2 \\
&=\max_{\substack{y_2\in \K^p\\ y_2\neq 0}} \|(B_{\M}B_{\M}^+\otimes I_{n^2}) K^{(2)}_f(A)B_{\M}y_2\|_2  \\
& = \|(B_{\M}B_{\M}^+\otimes I_{n^2}) K^{(2)}_f(A)B_{\M}B^+_{\M}\|_2,
\end{align*}
again using $B_{\M}=B_{\M}B_{\M}^+B_{\M}$.
\end{proof}

We now modify \cite[Algorithm~4.2]{hire14a} to obtain an upper bound for the structured level\=/2 condition number. It requires any algorithm to compute $L_f(A,E_1,E_2)$ for $f:\S_{\M} \rightarrow \S_{\N}$, where $\S_{\M} \subseteq \K^{\nbyn}$ is a smooth matrix manifold and $E_i\in T_A\S_{\M}$, i.e., $\vec(E_i)=B_{\M}y_i$ for some $y_i\in \K^p$, and returns $\gamma=\|(B_{\M}B_{\M}^+\otimes I_{n^2}) K^{(2)}_f(A)B_{\M}B^+_{\M}\|_2$. The resulting method is depicted in Algorithm~\ref{algor1}.

\begin{algorithm}
\caption{Upper bound for structured level\=/2 condition number}\label{algor1}
\begin{algorithmic}[1]
\setstretch{1.2}

\Statex \textbf{Input:} Matrix $A\in\R^{n\times n}$, function $f$, basis $B_{\M}$ of $T_A \S_{\M}$
\Statex \textbf{Output:} Upper bound $ \gamma=\|(B_{\M}B_{\M}^+\otimes I_{n^2}) K^{(2)}_f(A)B_{\M}B^+_{\M}\|_2$ for structured level\=/2 condition number

\smallskip
\State $\Psi \leftarrow 0_{n^2\times p}$ 
\State $\Omega \leftarrow 0_{n^4\times p}$ 
\For{$i=1:p$}
\State Choose $E_1\in \K^{\nbyn}$ such that $\vec(E_1)=B_{\M}e_i$ 
\For{$j=1:p$}
\State Choose $E_2\in \K^{\nbyn}$ such that $\vec(E_2)=B_{\M}e_j$ 
\State Compute $L=L^{(2)}_f(A,E_1,E_2)$ using the relation (\ref{eq.higherfrechet})
\State $\Psi e_j \leftarrow \vec(L)$  
\EndFor
\State $\Omega e_i \leftarrow \vec(\Psi)$ 
\EndFor
\State $\gamma=\|\Omega\|_2$

\end{algorithmic}
\end{algorithm}

We brief\/ly comment on the computational cost of Algorithm~\ref{algor1}: For typical practically relevant matrix functions, algorithms are available for evaluating them at a cost of $\mathcal{O}(n^3)$. These also directly transfer (via the relation~\eqref{eq.higherfrechet}) to algorithms for computing~$L_f(A,E_1,E_2)$ at cost $\mathcal{O}(n^3)$, albeit with an additional constant $4^3 = 64$ hidden in the $\mathcal{O}$, as the matrix $X_2$ in~\eqref{eq.higherfrechet} is of size $4n \times 4n$. Thus, the overall cost of the algorithm is $\mathcal{O}(p^2n^3)$, where $p$ is the dimension of the tangent space $T_A \S_{\M}$ and the constant can be expected to be quite large. This makes the algorithm prohibitively expensive even for medium scale problems (a typical limitation that is also already present when computing unstructured level\=/2 or even level-1 condition numbers). 

The computational cost of individual \F derivative evaluations can be reduced somewhat by using complex-step approximations~\cite{al2021complex} or quadrature-based algorithms~\cite{schweitzer2023}. While in general, when $A$ and the $E_i$ are unstructured, this does not change the asymptotic scaling of Algorithm~\ref{algor1}, experiments in~\cite{schweitzer2023} indicate that run time can be reduced by about one order of magnitude. If the direction terms $E_i$ are rank-one matrices (which, e.g., happens if the columns of $B_{\M}$ are certain canonical unit vectors), then run time of Algorithm~\ref{algor1} can potentially be reduced to $\mathcal{O}(p^2n^2)$ if linear systems with $A$ can be solved efficiently; see~\cite[Section~4.3]{schweitzer2023} for details.

\subsection{The level\=/2 condition number of (quasi-)triangular matrices}\label{sec.quasitriangular}
Another type of matrix structure that differs from those considered in section~\ref{subsec.structured_case} and~\ref{subsec:struc_lvl2} is (quasi-)triangularity, which plays an important role in many algorithms for computing matrix functions. 
An upper quasi-triangular matrix is a block upper triangular matrix with diagonal blocks of size at most $2 \times 2$.
These commonly occur when using the Schur decomposition: For general $A$, a \emph{Schur decomposition} $A = QUQ^*$ with $Q$ unitary and $U$ upper triangular is always guaranteed to exist. Note, however, that $U$ and $Q$ may be complex-valued even if $A$ is real. In the case of real-valued $A$, one can instead use a \emph{real Schur decomposition} $A = QUQ^T$, where the matrices $U$ and $Q$ are both real-valued, but $U$ is now upper quasi-triangular.

Given a (real) Schur decomposition, the relation $f(A) = Qf(U)Q^*$ shows that evaluating $f(A)$ can essentially be reduced to the evaluating $f$ at a (quasi-)triangular matrix. This forms the basis of many numerical methods for computing $f(A)$ at small and dense matrices $A$: the Schur--Parlett algorithm~\cite{dahi03,hiliu21}, for example. 

Structured level-1 condition numbers of $f(U)$ have recently been investigated in~\cite{almohy22}, 
and we extend this concept to the level\=/2 case. 
Quasi-triangular matrices do not fit into the framework of 
Lie and Jordan algebras or automorphism groups induced by a scalar product, 
and have an arguably simpler structure. We recall some important properties:
\begin{enumerate}
\item The set $\S$ of quasi-triangular matrices (with a fixed structure of diagonal blocks) forms a subalgebra of $\K^{n \times n}$ and is thus in particular a linear space and a smooth matrix manifold.
\item As $\S$ is a linear space, we have $T_U\S = \S$ for any $U \in \S$.
\item As any matrix function $f(U)$ is a polynomial in $U$ and $\S$ is a subalgebra of $\K^{n \times n}$, we have $f: \S \rightarrow \S$, i.e., a function of a quasi-triangular matrix is a quasi-triangular matrix with the same diagonal block structure.
\end{enumerate}

In light of the above observations, Lemma~\ref{lem.struclvl2} also applies to the case of quasi-triangular matrices, and for $B_\M$ we can take a basis of the space of upper quasi-triangular matrices (with the induced block structure). Such a basis is easily constructed from canonical unit vectors: Using the notation introduced in~\cite{almohy22}, let $d \in \R^{n-1}$ denote a vector that encodes the diagonal block structure of $U$, i.e.,
\[
d_i = \begin{cases}
\ 1, & \text{if } u_{i+1,i} \neq 0 \\
\ 0, & \text{otherwise},
\end{cases}
\qquad \text{ for } i = 1,\dots,n-1.
\]
To clearly indicate that the block structure is fixed, we denote the algebra of all quasi-triangular matrices with block structure encoded by $d$ as $\S_d$. Clearly, $\dim(\S_d) = \frac{n(n+1)}{2} + \nnz(d)$, where $\nnz$ denotes the number of nonzero entries. A (vectorised) basis of $\S_d$ is given by
\begin{equation}\label{eq.basis_triangular}
\{e_j \otimes e_i : j = 1,\dots,n, i = 1,\dots,j\} \cup \{ e_j \otimes e_{j+1} : j = 1,\dots,n-1 \text{ with } d_j = 1\}.
\end{equation}
The basis~\eqref{eq.basis_triangular} is orthonormal by construction, so if we take these basis vectors as columns of $B_\M$, we have $B_\M^+ = B_\M^*$.

\subsection{Exact structured level\=/2 condition numbers for special cases}\label{subsec:exact}
So far, we were only able to determine upper bounds for the structured level\=/2 condition number, which is in general all one can hope to obtain. However, for a few particular combinations of matrix class, function $f$ and used norm, it is possible to give an exact, closed form for $\cond_{\struc}^{[2]}(f,A)$.

\paragraph{Exponential of skew-Hermitian matrices}
When using the spectral norm (instead of the Frobenius norm that is typically employed), one can exactly compute the structured level\=/2 condition number when $A$ comes from the Lie Algebra of skew-Hermitian matrices. Note that exponentials of skew-Hermitian matrices, e.g., play an important role in the simulation of quantum systems, where according to the Schrödinger equation, the evolution of states is governed by $\exp(iH)$, where $H$ is the (Hermitian) Hamiltonian of the system; see, e.g.,~\cite{michel2022timeevolver}. As a special case of the very general result that the exponential map takes a Lie algebra into its corresponding Lie group, it is well-known that the exponential of a skew-Hermitian matrix is unitary, a property that is fundamental for proving the following result.

\begin{proposition}\label{prop:exp_skew}
Let $A \in \C^{\nbyn}$ be skew-Hermitian. Then, in the spectral norm,
\[
\cond_{\struc}^{[2]}(\exp,A) = 0.
\]
\end{proposition}
\begin{proof}
By a result from~\cite[Section~4]{van1977sensitivity} for the unstructured condition number, it holds for any normal matrix $M$ that in the spectral norm $\cond^{[1]}(\exp,M) = e^{\alpha(M)}$, where $\alpha(M)$ denotes the largest real part of any eigenvalue of $M$ (the so-called \emph{spectral abscissa}). Any skew-Hermitian matrix $A$ is normal and has eigenvalues on the imaginary axis, so that $\alpha(A) = 0$, and therefore $\cond^{[1]}(\exp,A) = e^0 = 1$. We now argue that the same holds for the structured level-1 condition number, i.e., $\cond_{\struc}^{[1]}(\exp,A) = 1$. Consider the skew-Hermitian matrix $Y_\varepsilon = A+i\varepsilon I$. We then have $\exp(Y_\varepsilon) = e^{i\varepsilon}\exp(A)$ and $\|Y_\varepsilon-A\| = \varepsilon$. Taking this into account, we find
\begin{equation}\label{eq.Yeps}
\lim\limits_{\varepsilon \rightarrow 0} \frac{\|\exp(Y_{\varepsilon})-\exp(A)\|}{\varepsilon} = \lim\limits_{\varepsilon \rightarrow 0}\frac{\| (e^{i\varepsilon}-1)\exp(A) \|}{\varepsilon} = \lim\limits_{\varepsilon \rightarrow 0}\frac{|e^{i\varepsilon}-1| \cdot \|\exp(A)\|}{\varepsilon} = \lim\limits_{\varepsilon \rightarrow 0}\frac{|e^{i\varepsilon}-1| }{\varepsilon} = 1,
\end{equation}
where we have used in the second to last equality that $\exp(A)$ is unitary and therefore has spectral norm equal to one. From~\eqref{eq.Yeps} it follows that the supremum in~\eqref{eq.defstruc} is at least one. On the other hand, the structured condition number is always bounded above by the unstructured one, so they must agree. Using this, the assertion of the proposition directly follows: From the definition~\eqref{eq.defstruclevel2}, we have
\[
\cond_{\struc}^{[2]}(\exp,A)=\lim_{\epsilon\rightarrow 0}\:\:\sup_{\substack{A+Z\in \S_{\M}\\ \| Z\|\leq \epsilon}}\, \dfrac{|\cond_{\struc}^{[1]}(\exp,A+Z)-\cond_{\struc}^{[1]}(\exp,A)|}{\epsilon}=\lim_{\epsilon\rightarrow 0}\:\:\sup_{\substack{A+Z\in \S_{\M}\\ \| Z\|\leq \epsilon}}\, \dfrac{|1-1|}{\epsilon} = 0,
\]
where we have used that $A+Z$ is again skew-Hermitian.
\end{proof}

\paragraph{Exponential of Hermitian matrices}
Next, we investigate the Hermitian case. As in the skew-Hermitian case, the structured and unstructured level-1 condition numbers agree, and this time, this carries over to the level\=/2 condition numbers.
\begin{proposition}\label{prop:exp_herm}
Let $A \in \C^{\nbyn}$ be Hermitian. Then, in the spectral norm,
\[
\cond_{\struc}^{[2]}(\exp,A) = \cond^{[2]}(\exp,A) = e^{\lambda_{\max}},
\]
where $\lambda_{\max}$ denotes the largest eigenvalue of $A$.
\end{proposition}
\begin{proof}
Proceeding as in the proof of Proposition~\ref{prop:exp_skew}, we have that $\cond_{\struc}^{[1]}(\exp,A) = \cond^{[1]}(\exp,A)  = e^{\alpha(A)} = e^{\lambda_{\max}}$. According to~\cite[Theorem~5.2]{hire14a}, it also holds that $\cond^{[2]}(\exp,A) = e^{\lambda_{\max}}$. Clearly, for any Hermitian $Y$ with $\|Y-A\| \leq \varepsilon$ we have $e^{\alpha(Y)} \leq e^{\lambda_{\max}+\varepsilon}$, as the spectral norm of a Hermitian matrix is simply its largest eigenvalue. The upper bound is attained for the Hermitian matrix $Y = A + \varepsilon I$, and it follows from the definition of the structured level\=/2 condition number that
\[
\cond_{\struc}^{[2]}(\exp, A) = \lim\limits_{\varepsilon \rightarrow 0} \frac{|e^{\lambda_{\max}+\varepsilon} - e^{\lambda_{\max}}|}{\varepsilon} = e^{\lambda_{\max}}.
\]
\end{proof}

\begin{remark}
One might argue that the results of Proposition~\ref{prop:exp_skew} and~\ref{prop:exp_herm} are of limited practical relevance. However, they highlight quite interesting properties and difficulties related to level\=/2 (structured) condition numbers: In both cases, the structured and unstructured level-1 condition numbers agree, but the implications for the level\=/2 condition number are fundamentally different. While in the Hermitian case, the structured and unstructured level\=/2 condition number also always agree, they can \emph{never} be equal in the skew-Hermitian  case (the level\=/2 unstructured condition number of the exponential of a skew-Hermitian matrix is always equal to $1$ according to~\cite[Theorem~5.2]{hire14a}). As the structured condition number is zero, this even shows that without additional assumptions, it is impossible to derive any meaningful bounds on the \emph{ratio} between the unstructured and structured condition number,
\[
\frac{\cond^{[2]}(f, A)}{\cond_{\struc}^{[2]}(f, A) } \leq C 
\]
with $C > 1$ for relating the structured and unstructured condition numbers (i.e., in general, the unstructured condition number can be ``infinitely worse'' than the structured one).
\end{remark}

\subsection{Lower bounds for the unstructured level\=/2 condition number}
\label{subsec.lower_bounds}
One fundamental difficulty in gauging the quality of the results obtained in sections~\ref{subsec:struc_lvl2} and~\ref{sec.quasitriangular} is that there is no known algorithm to compute the exact level\=/2 condition number (in either the structured or unstructured case) 
except for very particular special cases where analytic forms are known; see the discussion at the end of section~\ref{subsec.structured_case} for the unstructured case and section~\ref{subsec:exact} above for the structured one. Thus, instead of comparing the structured level\=/2 condition number to the unstructured level\=/2 condition number we can only compare upper bounds for both quantities.

This has a fundamental limitation in that we do not know the sharpness of the bounds: if an arbitrary upper bound for quantity $Q_1$ lies below an arbitrary upper bound for quantity $Q_2$, this does not tell us anything about the relation between $Q_1$ and $Q_2$. Thus, even when~\eqref{eq.upper_bound_structured} lies orders of magnitude below~\eqref{eq.upper_bound_unstructured}, this does not necessarily mean that $\cond^{[1]}(f,A)$ and $\cond_{\struc}^{[1]}(f,A)$ are different. Of course, one can intuitively expect that both bounds~\eqref{eq.upper_bound_unstructured} and~\eqref{eq.upper_bound_structured} are similar in quality as they originate in strongly related concepts,  but this is not guaranteed. 

In order to obtain more robust results, we also compute lower bounds for the level\=/2 condition numbers in our experiments: If an upper bound for the structured condition number lies well below a lower bound for the unstructured condition number, then this is clearly also the case for the exact quantities, which must differ by \emph{at least} the same amount. To obtain lower bounds, observe that in light of~\eqref{eq.deflevel2}, for any particular matrix $Z$ with $\|Z\|_F = 1$ we clearly have
\begin{equation}\label{eq.lower_bound}
\cond^{[2]}(f,A) \geq \lim_{\epsilon\rightarrow 0}\:\: \dfrac{|\cond^{[1]}(f,A+\epsilon Z)-\cond^{[1]}(f,A)|}{\epsilon}.
\end{equation}
Thus, the right-hand side of~\eqref{eq.lower_bound} gives a lower bound of the unstructured condition number for any fixed~$Z$. Clearly, a similar relation also holds for the structured condition number, provided that $Z$ is chosen from the according tangent space.

To obtain lower bounds in our experiments, we use the MATLAB built-in function \texttt{fminsearch}, 
which implements the derivative-free simplex search method from~\cite{lare98} for maximising the right-hand side of~\eqref{eq.lower_bound} by choosing $Z$. 
Within our numerical experiments, 
we replace the limit and use a finite difference quotient instead, 
with $\epsilon = 10^{-3}$. 
Strictly speaking, what we obtain this way is only an \emph{approximate lower bound}, but a comparison to the upper bounds indicates that the result of the optimisation procedure is reliable.

\section{Numerical Experiments}
\label{sec.numexp}

\begin{figure}
\begin{minipage}{.48\linewidth}
\centering
\subfloat[Orthogonal matrices]{\includegraphics[scale=.53]{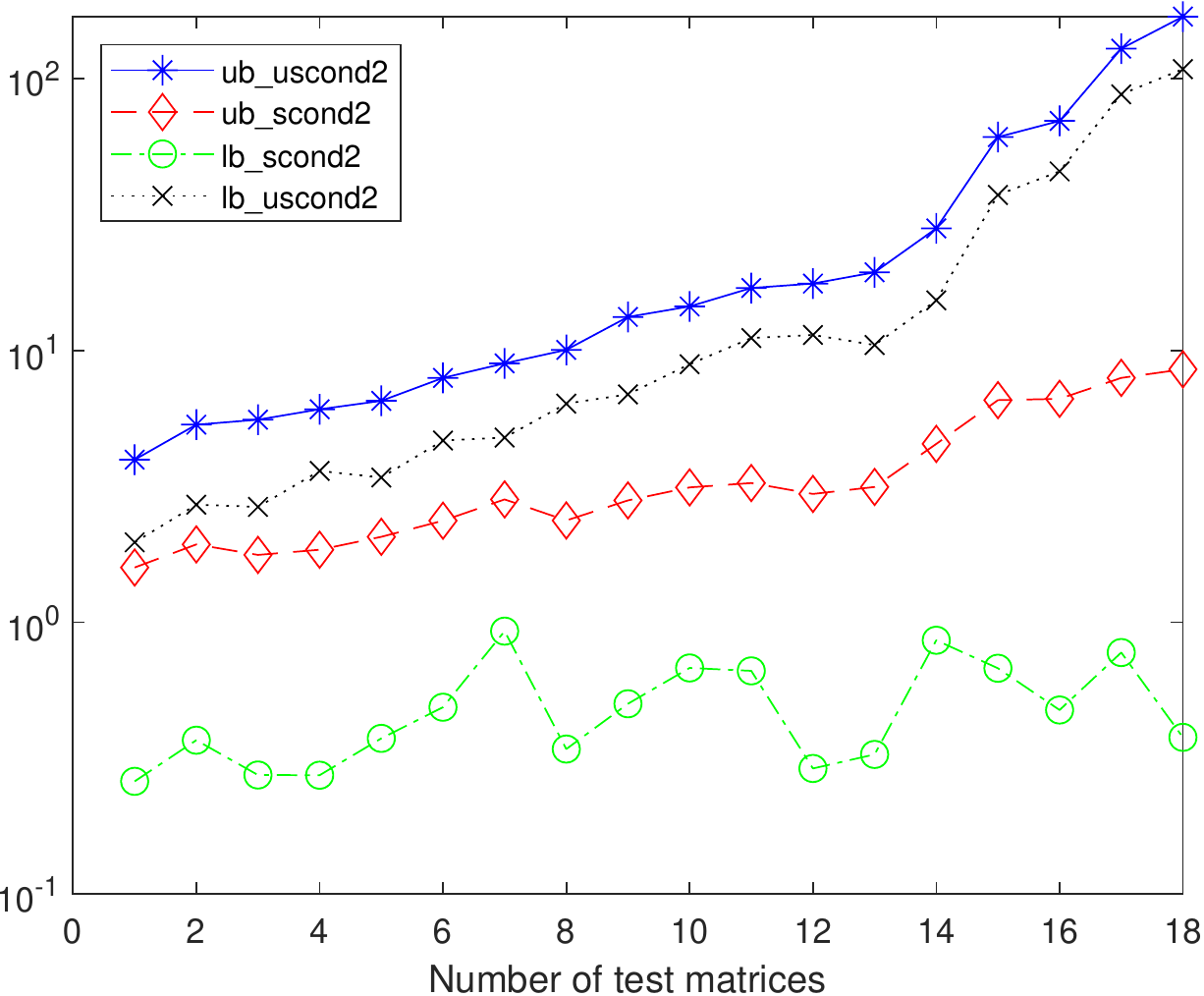}}
\end{minipage}
\begin{minipage}{.48\linewidth}
\centering
\subfloat[Symplectic matrices]{\includegraphics[scale=.53]{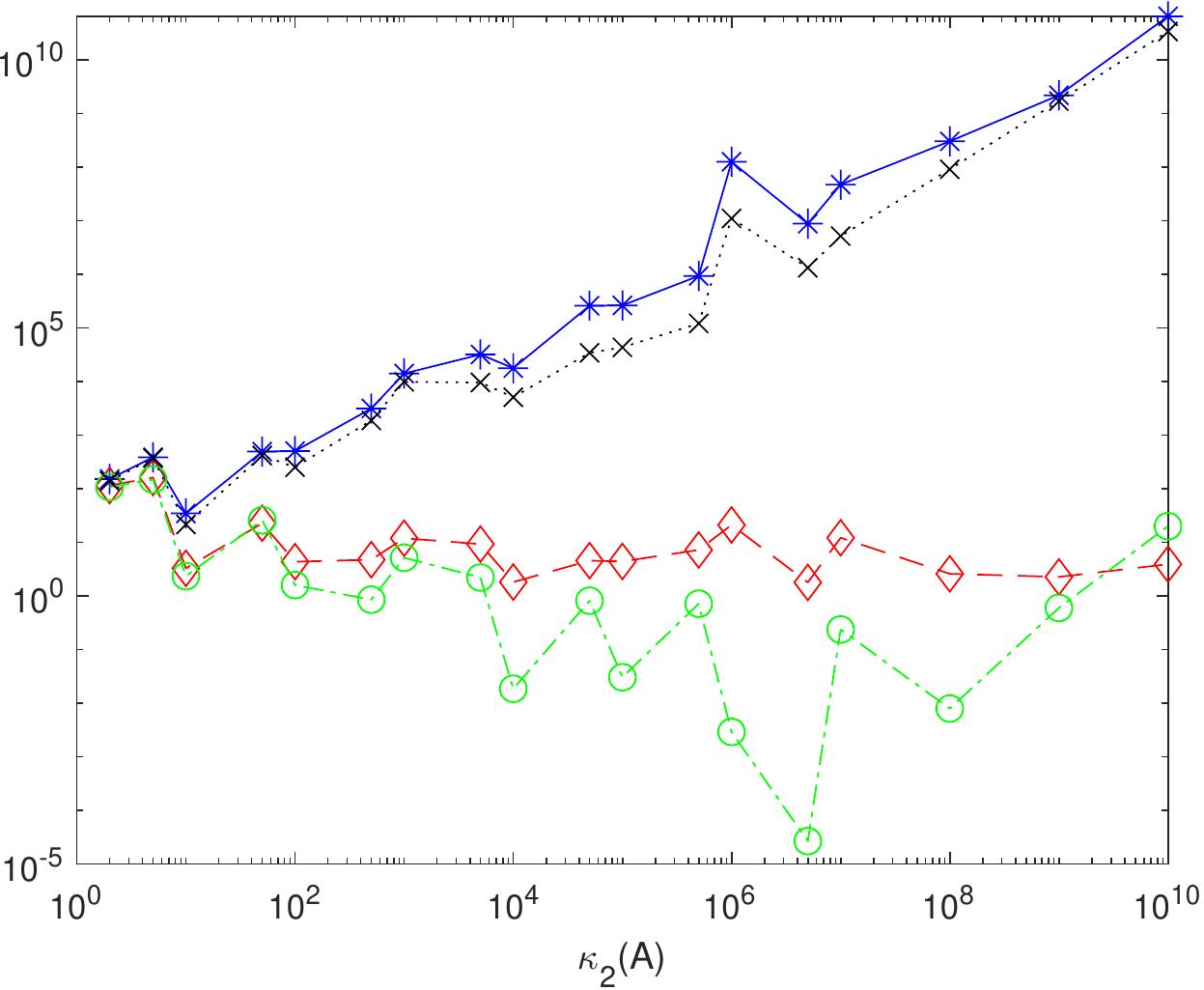}}
\end{minipage}\par\medskip
\centering
\subfloat[Perplectic matrices]{\includegraphics[scale=.53]{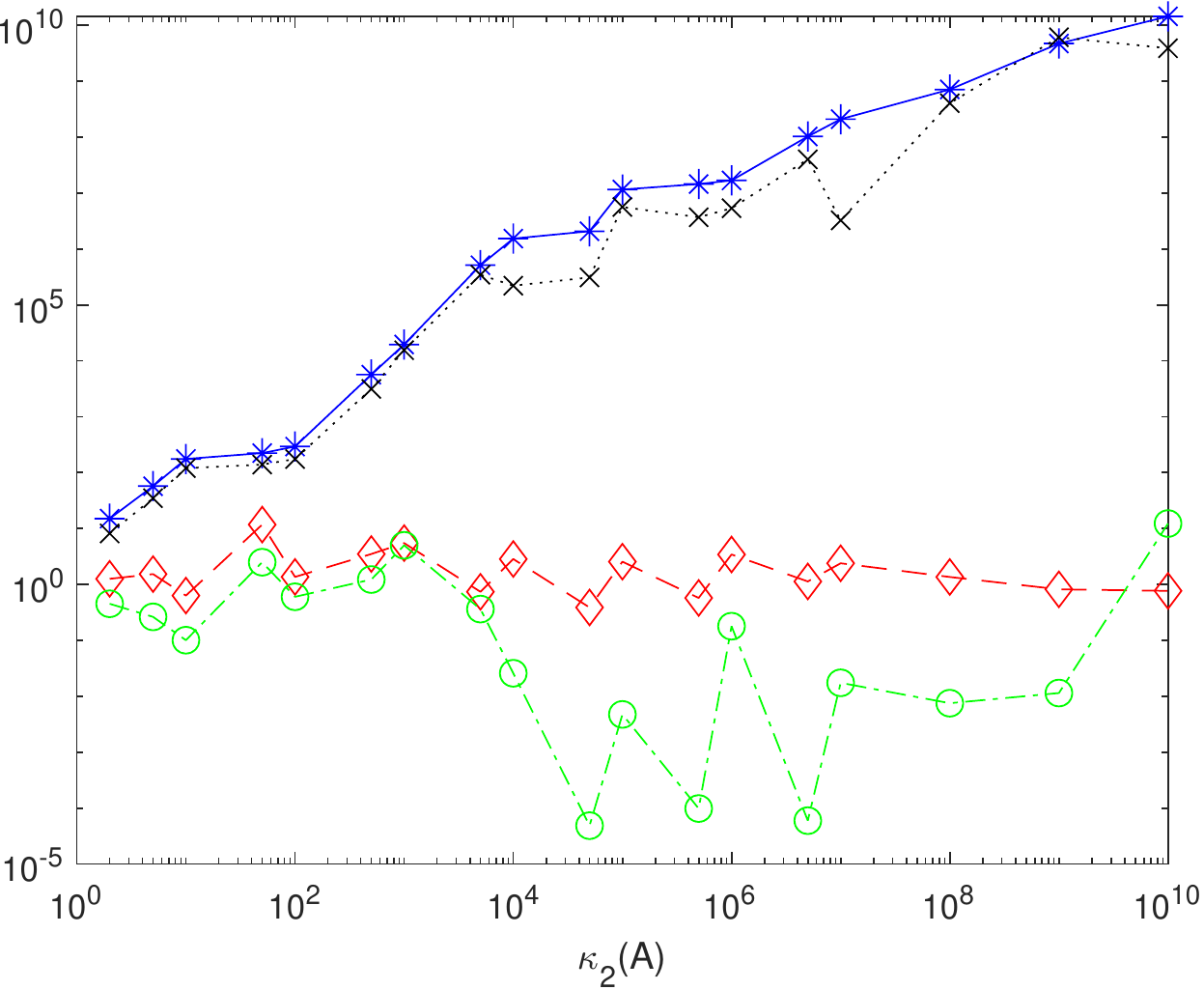}}
\caption{Structured and unstructured level\=/2 condition numbers for the matrix logarithm, comparing upper and lower bounds.}
\label{fig.main.logm}
\end{figure}

\begin{figure}
\begin{minipage}{.48\linewidth}
\centering
\subfloat[Orthogonal matrices]{\includegraphics[scale=.53]{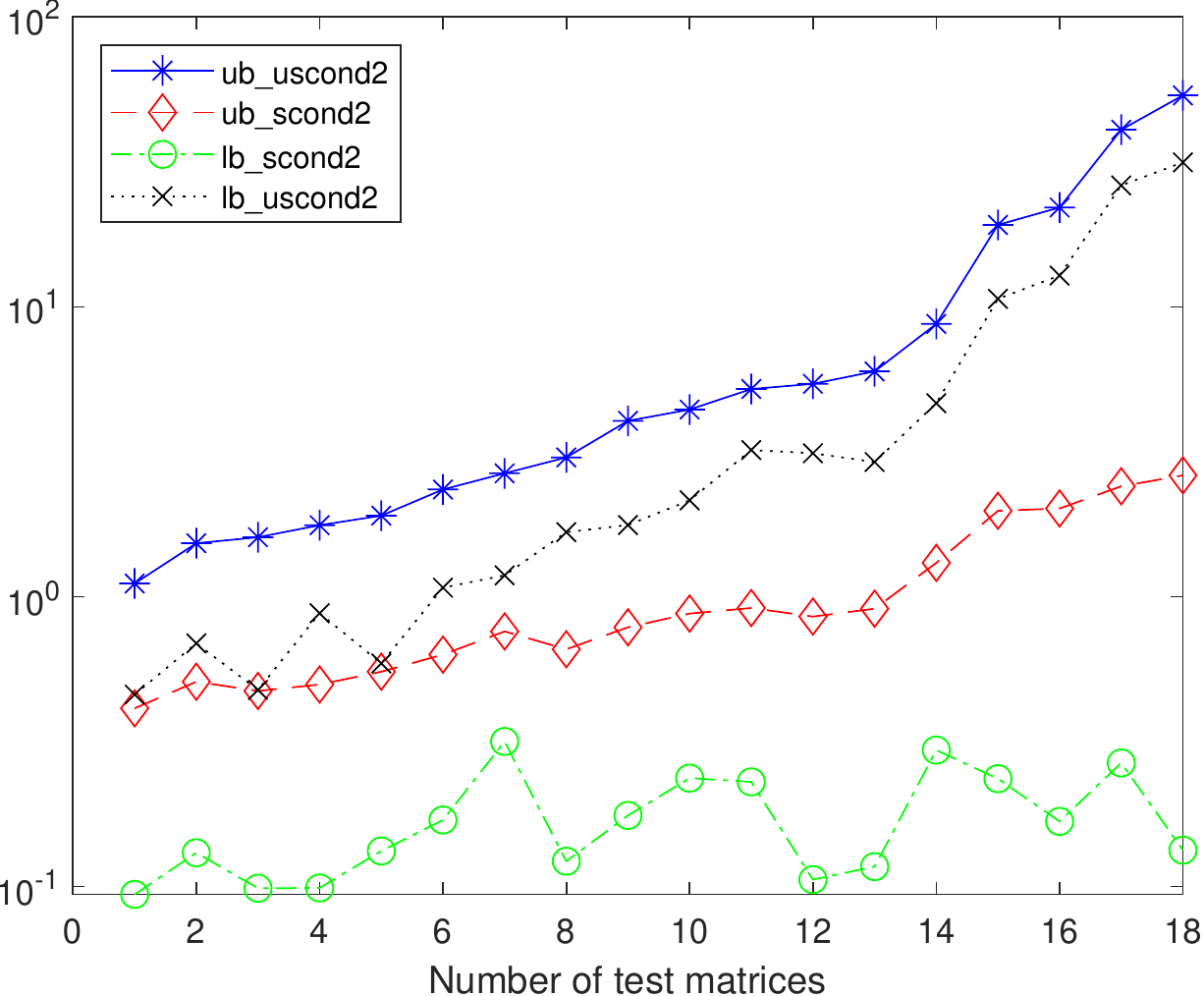}}
\end{minipage}
\begin{minipage}{.48\linewidth}
\centering
\subfloat[Symplectic matrices]{\includegraphics[scale=.53]{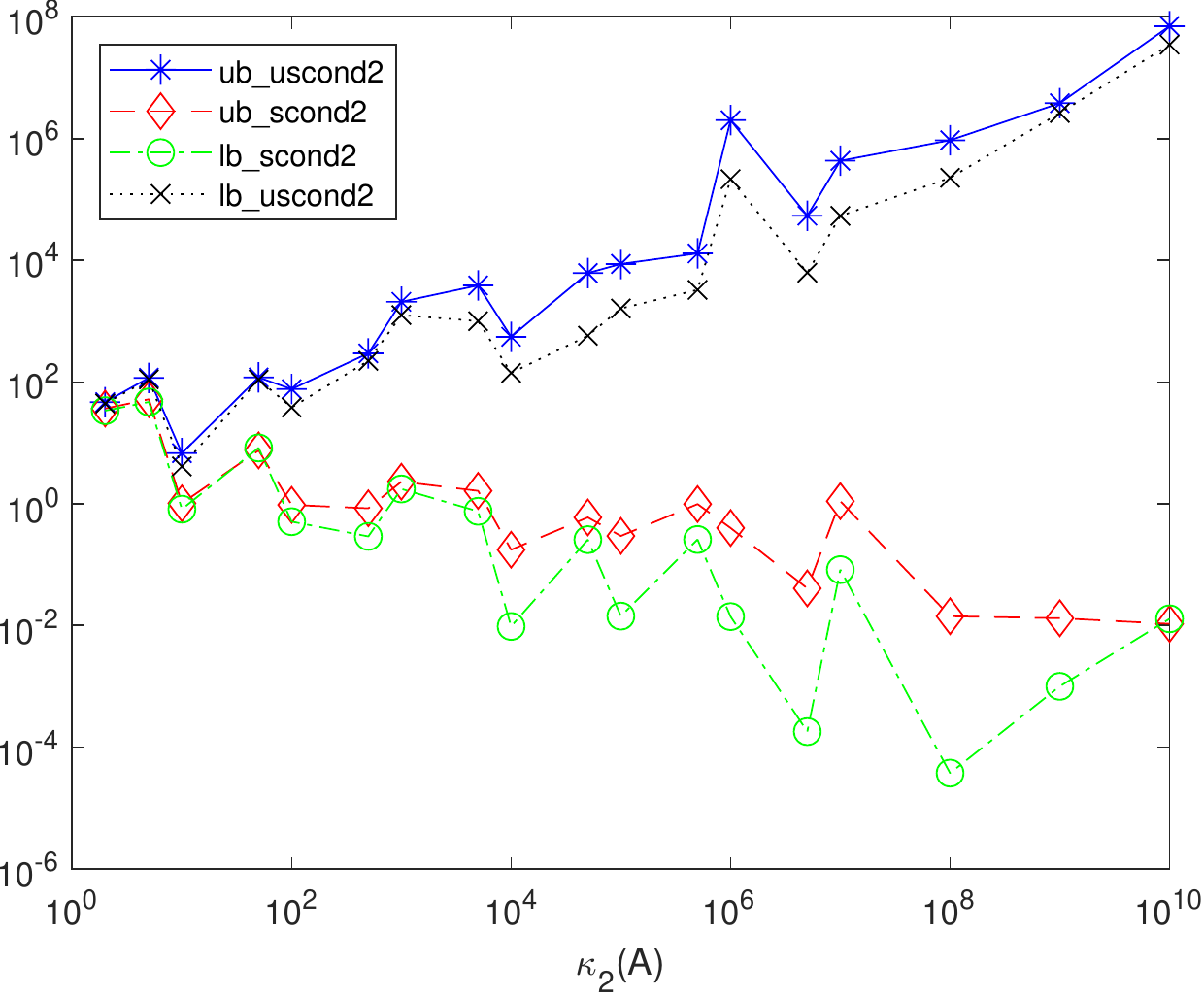}}
\end{minipage}\par\medskip
\centering
\subfloat[Perplectic matrices]{\label{fig.logm3}\includegraphics[scale=.53]{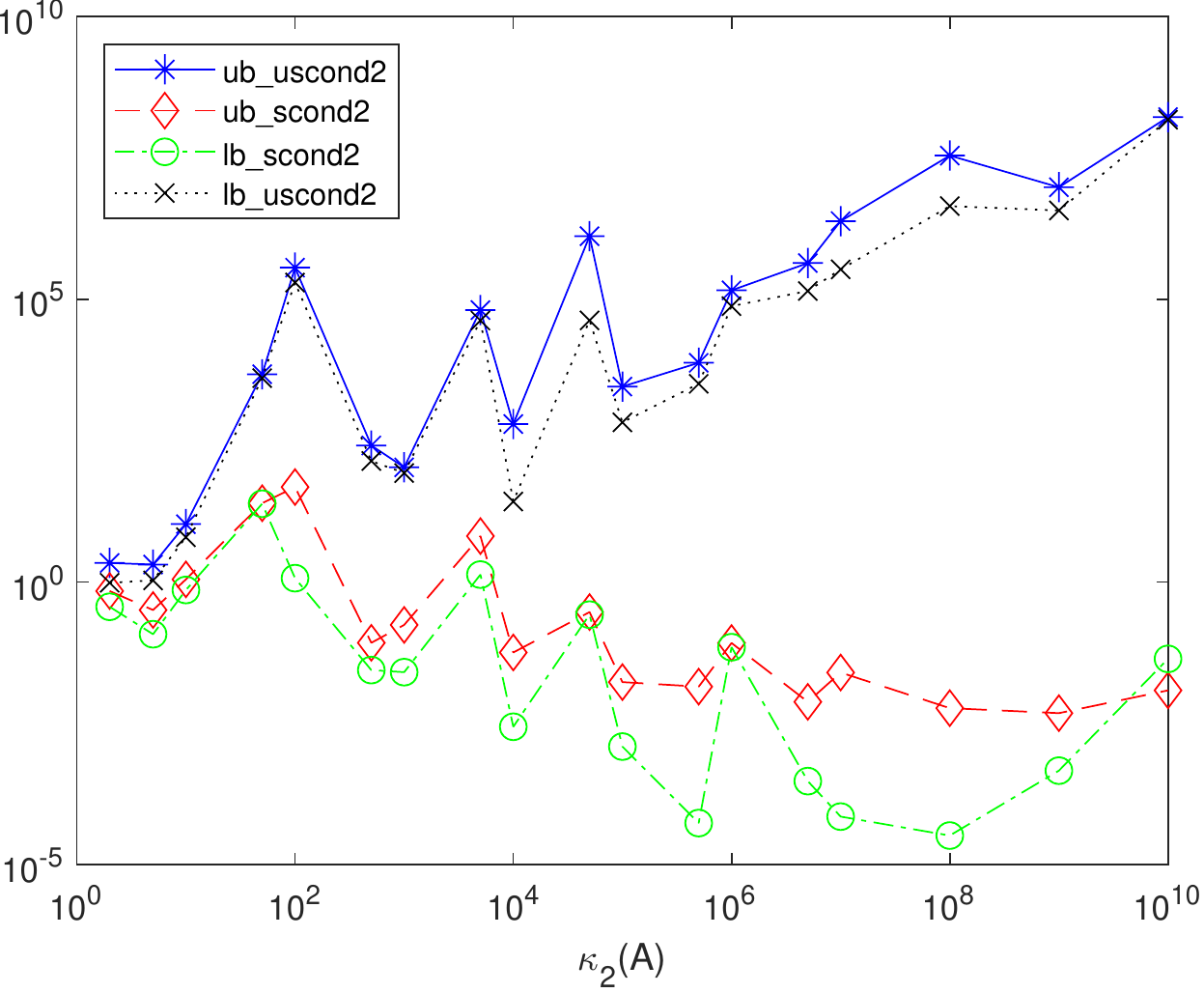}}
\caption{Structured and unstructured level\=/2 condition numbers for the matrix square root, comparing upper and lower bounds.}
\label{fig.main.sqrtm}
\end{figure}

\begin{figure}
\begin{minipage}{.48\linewidth}
\centering
\subfloat[Skew-symmetric matrices]{\includegraphics[scale=.53]{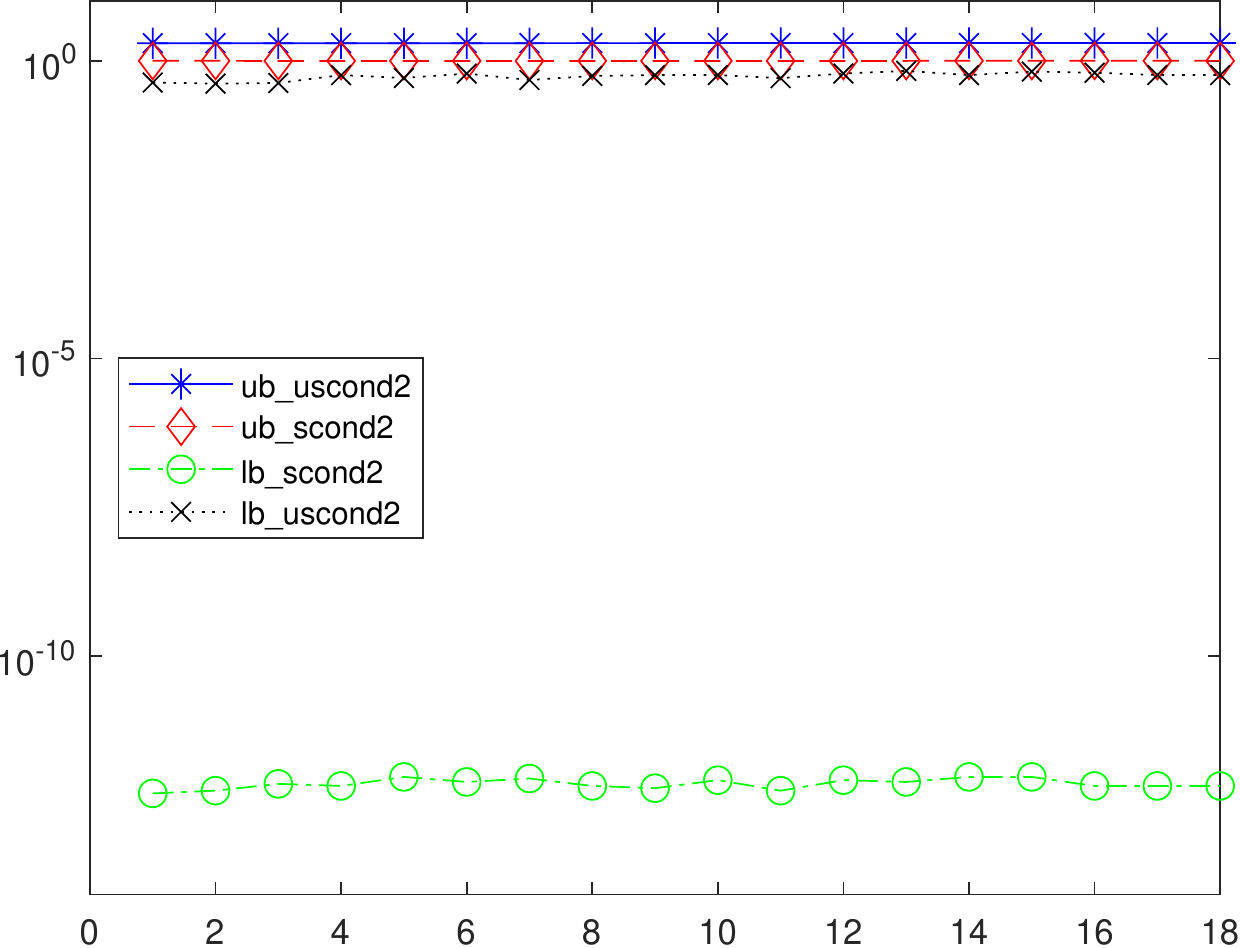}}
\end{minipage}
\begin{minipage}{.48\linewidth}
\centering
\subfloat[Hamiltonian matrices]{\includegraphics[scale=.53]{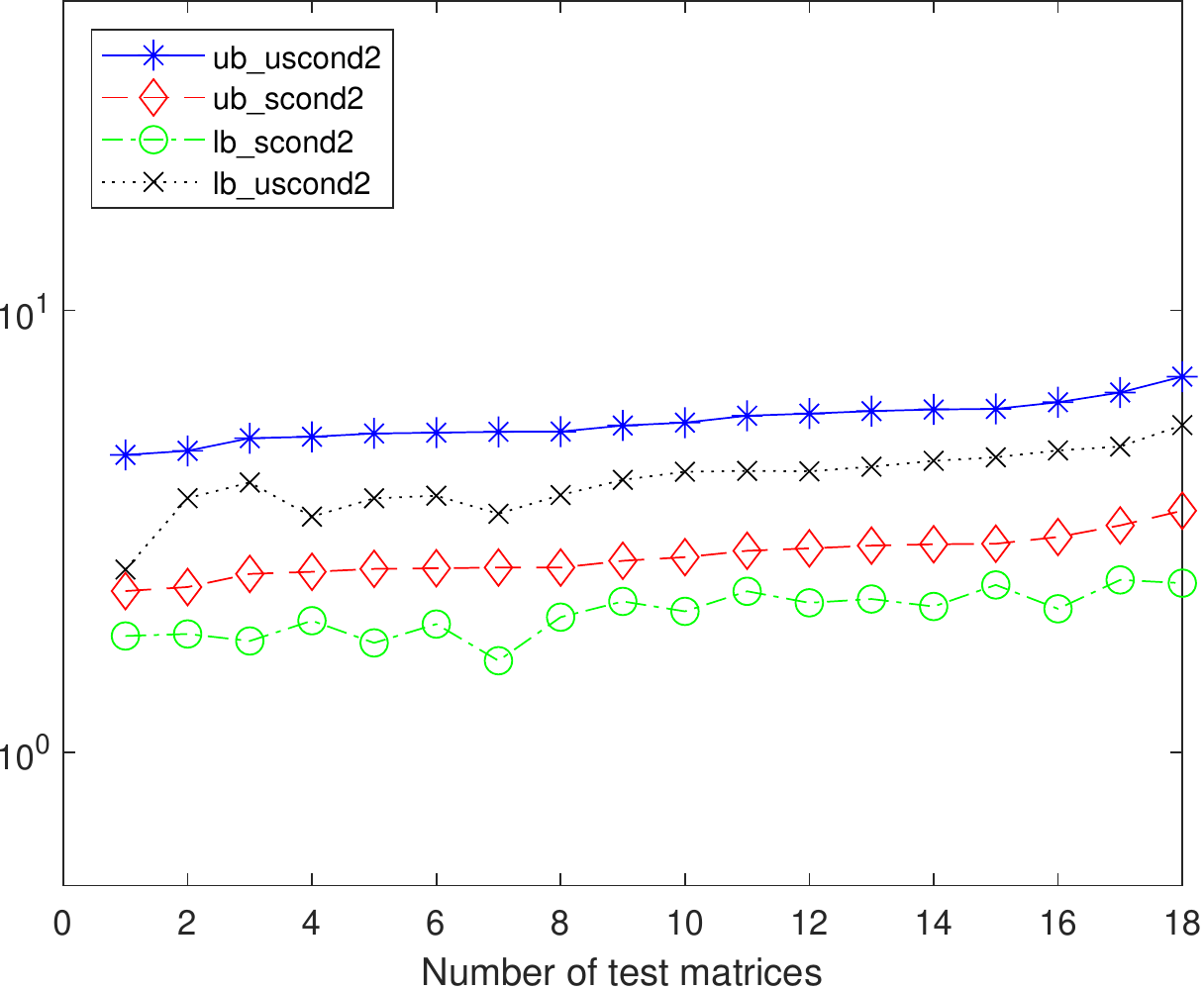}}
\end{minipage}
\caption{Structured and unstructured level\=/2 condition numbers for the matrix exponential, comparing upper and lower bounds.}
\label{fig.main.expm}
\end{figure}

\begin{figure}
\begin{minipage}{.48\linewidth}
\centering
\subfloat[Randomly generated matrices]{\includegraphics[scale=.53]{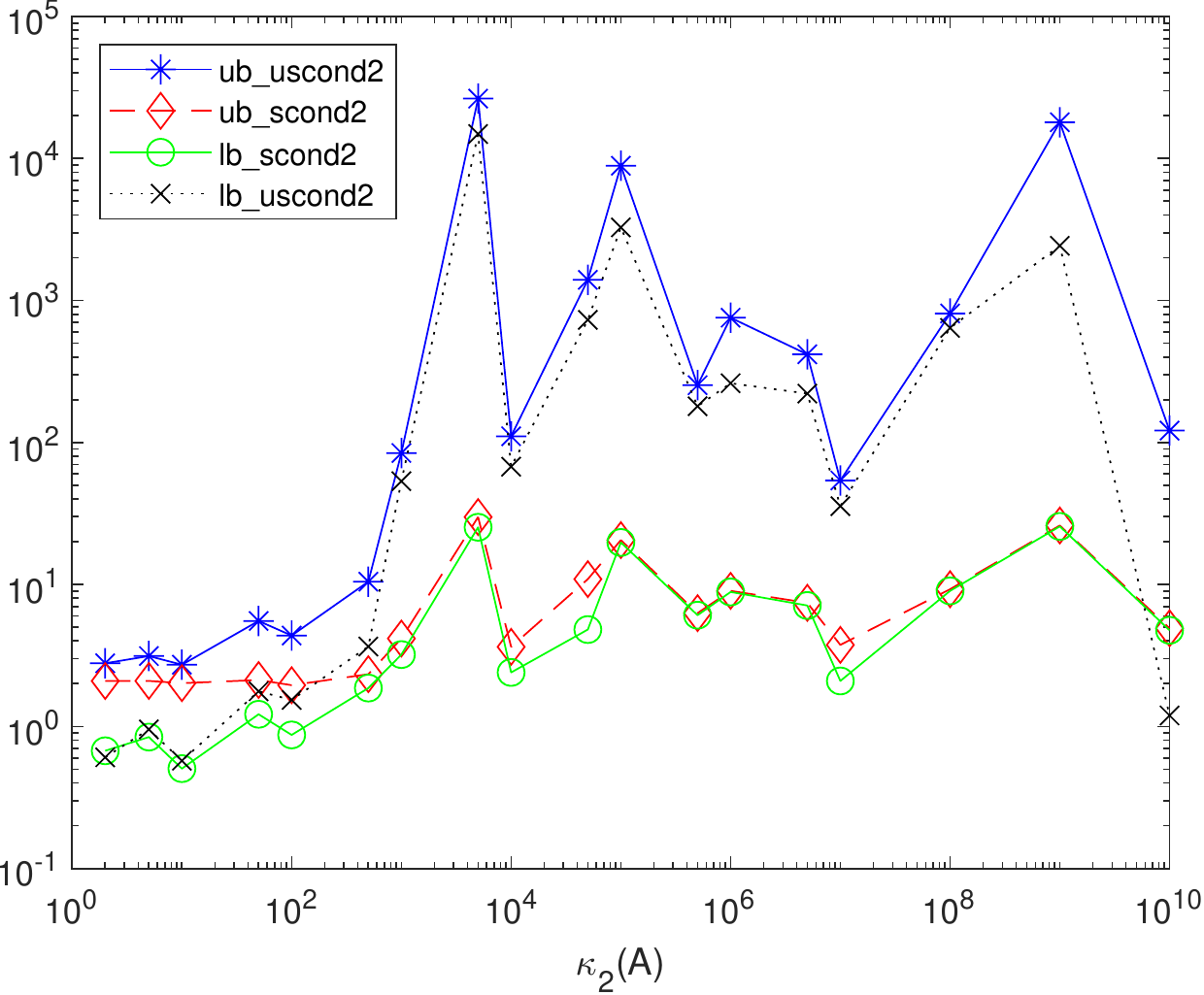}}
\end{minipage}
\begin{minipage}{.48\linewidth}
\centering
\subfloat[Matrices from benchmark sets]{\includegraphics[scale=.53]{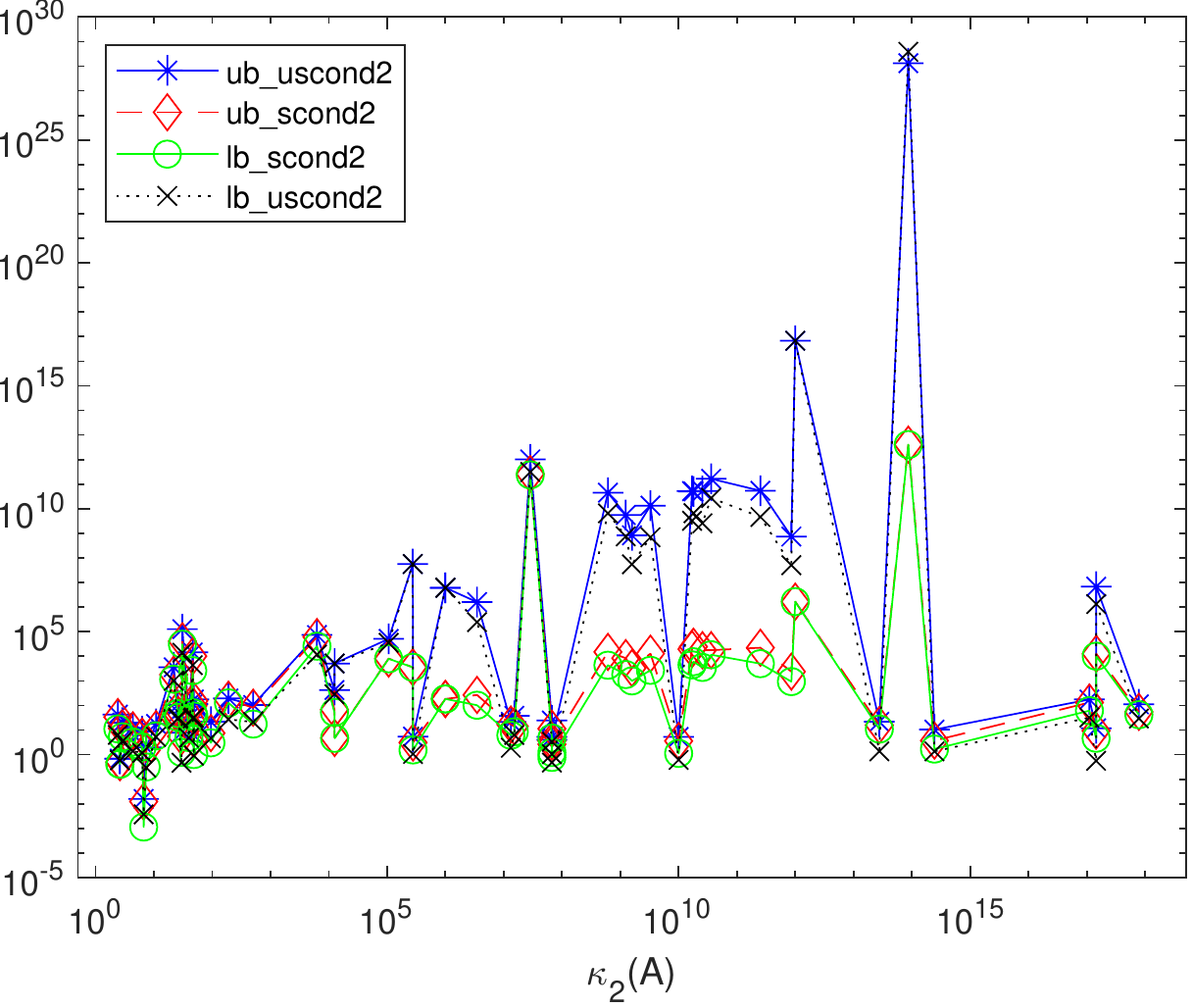}}
\end{minipage}
\caption{Structured and unstructured level\=/2 condition numbers for the exponential of upper quasi-triangular matrices, comparing upper and lower bounds.}\label{fig.triangular}
\end{figure}

The aim of this section is to compare the structured and unstructured level\=/2 condition numbers for the matrix logarithm, matrix square root and matrix exponential. As these cannot be computed exactly, we compare the upper bounds from sections~\ref{subsec.unstruc_lvl2},~\ref{subsec:struc_lvl2} and~\ref{sec.quasitriangular}, as well as the approximate lower bounds from section~\ref{subsec.lower_bounds}. 

We use MATLAB R2022a on a machine with an AMD Ryzen 7 3700X 8-core CPU to run the experiments.  We build test matrices as follows.
\begin{itemize}
\item
Orthogonal, symplectic and perplectic matrices are built using Jagger's toolbox \cite{jagg03}.
\item
Hamiltonian matrices are built as
$$
A=\begin{bmatrix}
X & G  \\
F & -X^T
\end{bmatrix},\quad
\text{where} \quad F^T=F\quad \text{and} \quad G^T=G.
$$
\item Quasi-triangular matrices are artificially generated as follows: We generate a diagonal matrix $D$ with equidistantly distributed entries in the interval $[-c, -1]$ for a specified parameter $c$. We then compute $A = XDX^{-1}$ with $X$ a randomly sampled matrix and perform a real Schur decomposition $A = QUQ^T$ to obtain the upper quasi-triangular matrix $U$.
\item Alternatively, we obtain quasi-triangular matrices by performing a (real) Schur decomposition of a large set of common benchmark matrices from the matrix function literature; see Appendix~\ref{sec.appendix} for details on the test set.
\end{itemize}
Depending on the matrix function we construct $B_{\M}$ using either equation $(\ref{eq.basejordanlie})$ or $(\ref{eq.baseauto})$. 

The upper bounds for the structured and unstructured level\=/2 condition numbers 
are computed using Algorithm \ref{algor1} and equation~(\ref{eq.upper_bound_unstructured}), respectively.
For symplectic, perplectic and quasi-triangular matrices the bounds are plotted against the 2-norm condition number for matrix inversion, $\kappa_2(A)=\|A\|_2\|A^{-1}\|_2$, whereas for orthogonal matrices we sort the test matrices by increasing values of the unstructured condition number.
Within the legend of each plot the acronyms \texttt{ub} and \texttt{lb} are used to indicate upper and lower bounds,
meanwhile \texttt{scond2} and \texttt{uscond2} denote the structured and unstructured level\=/2 condition numbers.

\begin{exprmt}
\item
We first consider the matrix logarithm $\log A:\G\rightarrow \L$ and by taking $\S_{\M}=\G$ we construct $B_{\M}$ using equation (\ref{eq.baseauto}). In Algorithm \ref{algor1} we compute the matrix logarithm by the MATLAB function \texttt{logm} which uses the improved inverse scaling and squaring method \cite{alhi12g}.
The upper bounds for the structured and unstructured level\=/2 condition numbers are shown in Figure \ref{fig.main.logm} for orthogonal, symplectic and perplectic matrices $A\in\R^{4\times 4}$. Comparing the results obtained from different test matrices, we observe that in particular in the symplectic and perplectic case, the structured condition number is much smaller than the unstructured one (in fact, the structured level\=/2 condition number appears to stay almost constant across all considered test matrices). In the orthogonal case, the difference is less pronounced, but can still be observed clearly. Both the magnitude as well as the slope at which it increases are substantially smaller for the structured level\=/2 condition number. 

The lower bounds that we also report confirm that in most cases, our upper bounds are quite tight and the structured condition number is indeed guaranteed to be much smaller than the unstructured one. Let us brief\/ly comment on the fact that for the most ill-conditioned symplectic and perplectic matrices, the lower bound for the unstructured condition number that we report lies \emph{above} the upper bound. This is likely caused by the fact that it is actually only an approximate lower bound (cf.~the discussion in section~\ref{subsec.lower_bounds}) combined with the severe ill-conditioning of the matrix under consideration. Apart from this, our results (also those reported in later experiments) still indicate that in general the lower bounds can be expected to be quite reliable.

\item
We take the matrix square root $A^{1/2}:\G\rightarrow \G$ and $B_{\M}$ is built by equation (\ref{eq.baseauto}). For computing the matrix square root we use \texttt{sqrtm} which is based on the method given in \cite{sqrtm_ref}. The bounds are shown in Figure \ref{fig.main.sqrtm} for orthogonal, symplectic and perplectic matrices $A\in\R^{4\times 4}$. We observe similar results as in the previous experiment. In particular, there is a dramatic increase in the unstructured level\=/2 condition numbers for ill-conditioned matrices, while the structured one again stays roughly constant for symplectic and perplectic matrices and only increases moderately for orthogonal matrices.

The lower bounds again confirm that our results are reliable and show that the upper bounds are quite tight, in particular in the symplectic and perplectic case.
\item
In this experiment we consider the matrix exponential $e^A:\L\rightarrow \G$ for skew-symmetric and Hamiltonian matrices $A\in\R^{4\times 4}$. We again construct $B_{\M}$ by equation (\ref{eq.basejordanlie}) and use the MATLAB function \texttt{expm}, which is based on a scaling and squaring algorithm \cite{alhi09a} to compute the exponential. Figure \ref{fig.main.expm} demonstrates that the upper bounds for the structured and unstructured level\=/2 condition number behave very similarly in both cases, with the bounds for the unstructured condition number being almost exactly two times as large as the structured one for both kinds of test matrices. In light of Proposition~\ref{prop:exp_skew} (although this result uses the spectral norm and not the Frobenius norm as in this experiment), it is interesting to observe that in the skew-symmetric case, the lower bound stays in the order of $10^{-13}$. This value close to machine precision appears to confirm that the exact value is zero.
\item 
We take randomly generated (as described at the beginning of this section) upper quasi-triangular matrices $U\in \R^{10 \times 10}$, with the parameter $c$ ranging between $2$ and $10^{10}$ and $f(U) = \exp(U)$. Note that we increase the matrix size compared to the previous experiments to allow for more variety in the diagonal block structure of $U$. We compute the upper bounds~\eqref{eq.upper_bound_unstructured} and~\eqref{eq.upper_bound_structured} as well as lower bounds for the structured and unstructured level\=/2 condition number. The results are depicted on the left-hand side of Figure~\ref{fig.triangular}. As it was the case for the other considered matrix structures, the bounds again confirm that structured level\=/2 conditioning can be much better than unstructured conditioning, in particular for the more ill-conditioned examples. The lower bounds obtained by optimisation are also mostly very tight.
\item 
In the final experiment, instead of randomly generating matrices, we compute bounds for the structured and unstructured level\=/2 condition numbers of $\exp(U)$, where $U$ is taken as the upper quasi-triangular factor of the (real) Schur decomposition of a wide range of matrices from the matrix function literature; see Appendix~\ref{sec.appendix} for details on the test set, which contains 56 matrices in total. All considered matrices are non-normal, so that the matrix $U$ from their Schur decomposition is not diagonal. For matrices that are scalable in size (which is, e.g., the case for most matrices coming from built-in MATLAB functions), we again choose them to be from $\R^{10 \times 10}$ (or the closest possible value if there are certain restrictions on the matrix size $n$). The resulting condition number bounds are shown on the right-hand side of Figure~\ref{fig.triangular}. In particular for those test matrices which are rather well-conditioned, the upper bounds for structured and unstructured condition numbers are very similar and differ by less than one order of magnitude. However, for a large portion of the more ill-conditioned examples, the upper bound of the structured condition number is often several orders of magnitude lower than that for the unstructured one. In almost all cases, the lower bounds confirm that the upper bounds are quite tight for those matrices, indicating that conclusions drawn from those bounds are indeed reasonable. Note that for a few of the more well conditioned matrices, the lower bound for the unstructured condition number actually lies below the upper bound for the structured condition number, so that no guaranteed conclusions can be drawn. As the results suggest that both condition numbers are close to each other anyway, this is not an essential problem, though. \label{ex.quasitriangular}
\end{exprmt}

\section{Conclusions}
\label{sec.conclusions}
This work compares upper and the lower bounds of structured level\=/2 condition numbers of matrix functions with unstructured ones. Our results show that the difference between the bounds of the structured and unstructured level\=/2 condition numbers depends on the choice of the matrices and matrix functions: While for orthogonal matrices, the upper bound of the structured level\=/2 condition number is not significantly smaller than the upper bound of the unstructured one, the importance of structure preserving algorithms emerges for the matrix logarithm and for the matrix square root of matrices in automorphism groups. Further analysis can be done for other matrix functions and matrix factorizations. Additionally, finding more efficient algorithms for computing the level\=/2 condition number would be very important for making the concept practically usable in actual computations.

\paragraph{Acknowledgements} We are grateful to Awad Al-Mohy for providing us his test suite of quasi-triangular matrices from~\cite{almohy22} for use in our Experiment~\ref{sec.numexp}.\ref{ex.quasitriangular}.

\appendix
\section{Details on test matrices used in Experiment~\ref{sec.numexp}.\ref{ex.quasitriangular}}\label{sec.appendix}
In this section, we provide a comprehensive list of the matrices that have been used in Experiment~\ref{sec.numexp}.\ref{ex.quasitriangular}, i.e., for comparing structured and unstructured condition numbers of quasi-triangular matrices. Most of the test matrices are available via built-in MATLAB routines, but some also come from other toolboxes and benchmark collections. In general, for a test matrix $A$ from some benchmark collection, we first compute the Schur decomposition $A = QUQ^*$ (or the real Schur decomposition $A = QUQ^T$ if $A$ is real),  and then use the upper triangular factor $U$ as input for the condition number estimation algorithms. 

\begin{table}[H]
    \centering
    \subfloat[]{%
    \begin{tabular}[t]{|l|l|l|l|l|}
    \hline
\texttt{binomial}$^\ast$	&  \texttt{forsythe}     &  \texttt{krylov}$^\ast$    & \texttt{rando}    \\    
\texttt{chebspec} 			&  \texttt{frank}        &  \texttt{leslie}           & \texttt{randsvd}  \\    
\texttt{chebvand}           &  \texttt{gearmat}      &  \texttt{lesp}             & \texttt{redheff}  \\    
\texttt{chow}            	&  \texttt{grcar}        &  \texttt{lotkin}           & \texttt{riemann}  \\    
\texttt{clement}   			&  \texttt{invhess}      &  \texttt{parter}           & \texttt{smoke}    \\    
\texttt{cycol}      		&  \texttt{invol}$^\ast$ &  \texttt{randcolu}         &                   \\
\texttt{dramadah}  			&  \texttt{kahan} 		 &  \texttt{randjorth}$^\ast$ &                   \\
    \hline
    \end{tabular}}
    \hspace{.5cm}
    \subfloat[]{%
    \begin{tabular}[t]{|l|}
    \hline
        \texttt{rand}               \\
        \texttt{randn}              \\
        \texttt{pascal}$^\ast$      \\
        \\
        \\
        \\
        \\
        \hline
    \end{tabular}}
    \hspace{.5cm}
    \subfloat[]{%
    \begin{tabular}[t]{|l|}
    \hline
        \texttt{gfpp}		\\     
        \texttt{makejcf}	\\  
        \texttt{rschur}   	\\
        \texttt{vand}  		\\
        \\
        \\
        \\
        \hline
    \end{tabular}}
    \caption{Test matrices used in Experiment~\ref{sec.numexp}.\ref{ex.quasitriangular}: (a) Matrices from MATLAB \texttt{gallery}, (b) other built-in MATLAB matrices, (c) Matrices from the Matrix Computation Toolbox~\cite{Higham:MCT}. ${}^\ast$: The entries of these matrices were scaled by $\frac{10}{\|A\|_1}$ to prevent overflow in the matrix exponential.}
    \label{tab.test_set}
\end{table}

Table~\ref{tab.test_set} summarises a large part of the matrices we used. Additionally, we used matrices which are not available via built-in MATLAB or toolbox functions:

The following two matrices are taken from~\cite[Test 3--4]{ward1977numerical},
\[
A_1 = \begin{bmatrix}
-131 & 19 & 18 \\
-390 & 56 & 54 \\
-387 & 57 & 52
\end{bmatrix},
\quad A_2 = {\small\begin{bmatrix}
       0 & 1 & 0 & 0 & 0 & 0 & 0 & 0 & 0 & 0 \\
       0 & 0 & 1 & 0 & 0 & 0 & 0 & 0 & 0 & 0 \\
       0 & 0 & 0 & 1 & 0 & 0 & 0 & 0 & 0 & 0 \\
       0 & 0 & 0 & 0 & 1 & 0 & 0 & 0 & 0 & 0 \\
       0 & 0 & 0 & 0 & 0 & 1 & 0 & 0 & 0 & 0 \\
       0 & 0 & 0 & 0 & 0 & 0 & 1 & 0 & 0 & 0 \\
       0 & 0 & 0 & 0 & 0 & 0 & 0 & 1 & 0 & 0 \\
       0 & 0 & 0 & 0 & 0 & 0 & 0 & 0 & 1 & 0 \\
       0 & 0 & 0 & 0 & 0 & 0 & 0 & 0 & 0 & 1 \\
10^{-10} & 0 & 0 & 0 & 0 & 0 & 0 & 0 & 0 & 0 \\
\end{bmatrix}}.
\]

The next two matrices are from~\cite[Example~3.10, Example~4.4]{dieci2000pade},
\[
A_3 = \frac{1}{8}{\small\begin{bmatrix}
1.3 & 1.3 & 1.3 & 1.3 & 10^6 & 10^6 & 10^6 & 10^6 \\
1.3 & 1.3 & 1.3 & 1.3 & 10^6 & 10^6 & 10^6 & 10^6 \\
1.3 & 1.3 & 1.3 & 1.3 & 10^6 & 10^6 & 10^6 & 10^6 \\
1.3 & 1.3 & 1.3 & 1.3 & 10^6 & 10^6 & 10^6 & 10^6 \\
0 & 0 & 0 & 0 & -1.3 & -1.3 & -1.3 & -1.3 \\
0 & 0 & 0 & 0 & -1.3 & -1.3 & -1.3 & -1.3 \\
0 & 0 & 0 & 0 & -1.3 & -1.3 & -1.3 & -1.3 \\
0 & 0 & 0 & 0 & -1.3 & -1.3 & -1.3 & -1.3
\end{bmatrix}}, \quad 
A_4 = \begin{bmatrix}
\exp(0.1) & 10^6\cdot\exp(0.1) \\
0         & 10^{-8} + \exp(0.1)
\end{bmatrix}.
\]

The following matrix is from~\cite[Section~4]{kenney1989condition},
\[
A_5 = \begin{bmatrix}
48  &  -49 &  50 &   49 \\
0   &  -2  & 100 &    0 \\
0   &  -1  &  -2 &    1 \\
-50 &  50  &  50 &  -52 
\end{bmatrix}.
\]

The next two matrices are adapted from~\cite[Example~II,~III]{parlett1985development},
\[
A_6 = {\small\begin{bmatrix}
- 3.5i	& -12    &   1	    & 1  	 & 	 1 	  &   1  &  1    &   1   \\
0		& - 2.5i &  -8 	    & 1  	 &   1 	  &   1  &  1    &   1   \\
0		&   0 	 &   - 1.5i & -4	 &   1 	  &   1  &  1    &   1   \\
0		&   0 	 &   0  	&- 0.5i  &   0 	  &   1  &  1    &   1   \\
0		&   0 	 &   0  	&  0     &   0.5i &   4  &  1    &   1   \\
0		&   0 	 &   0  	&  0     &   0    & 1.5i &  8    &   1   \\
0		&   0 	 &   0  	&  0     &   0    &   0  &  2.5i &  12   \\
0		&   0 	 &   0  	&  0     &   0    &   0  &  0 	 &  3.5i 
\end{bmatrix}},
\]
\[
A_7 = {\small\begin{bmatrix}
 -3.5 &  1   &   0   &  0     &  0   & 0    & 0   & 0  \\
    0 & -2.5 &   1   &   0    &  0   &  0   & 0   & 0  \\
    0 &  0   &  -1.5 &   1    &  0   &  0   & 0   & 0  \\
    0 &  0   &   0   &  -0.5  &  1   &  0   & 0   & 0  \\
    0 &  0   &   0   &   0    &  0.5 &  1   & 0   & 0  \\
    0 &  0   &   0   &   0    &  0   &  1.5 & 1   & 0  \\
    0 &  0   &   0   &   0    &  0   &  0   & 2.5 & 1  \\
    0 &  0   &   0   &   0    &  0   &  0   & 0   & 3  
\end{bmatrix}}.
\]

This matrix comes from~\cite[Test~2]{cardoso2001theoretical},
\[
A_8 = \begin{bmatrix}
-149 &  -50 & -154 \\
 537 &  180 &  546 \\
 -27 &   -9 &  -25
\end{bmatrix}.
\]

The next matrix is taken from~\cite[Example~6.3]{dieci1996computational},
\[
A_9 = \begin{bmatrix}
1 + 10^{-7} & 10^5      & 10^4 \\
0           & 1+10^{-1} & 10^5 \\
0           & 0         & 11   
\end{bmatrix}.
\]

Additionally, the following four matrices were provided to us by Awad Al-Mohy from his algorithm test set,
\[
A_{10} = \begin{bmatrix}
\exp(i(\pi-10^{-7})) & 1                   \\
0                    & \exp(i(\pi+10^{-7})) 
\end{bmatrix}, \quad
A_{11} = \begin{bmatrix}
\exp(i(\pi-10^{-7})) & 1000                \\
0                    & \exp(i(\pi+10^{-7})) 
\end{bmatrix},
\]

\[
A_{12} = \begin{bmatrix}
\exp(i(\pi-10^{-7})) & 1                                   \\
0                   & (1+10^{-7}i)\cdot\exp(i(\pi-10^{-7}))
\end{bmatrix},
\]
\[
A_{13} = \begin{bmatrix}
\exp(i(\pi-10^{-7})) & 1000                                \\
0                   & (1+10^{-7}i)\cdot\exp(i(\pi-10^{-7}))
\end{bmatrix},
\]

\bibliography{ref_level2}
\end{document}